\documentclass{amsart}
\usepackage{amssymb}
\usepackage{amsmath}
\usepackage{graphicx}
\usepackage{color}
\usepackage{hyperref} 
\def \C {{\Bbb C\,}}

\def\R {{\Bbb R}}
\def\S{{\Bbb S}}

\def \Z {{\Bbb Z}}
\def\Im{\mbox{\rm\,Im\,}}
\def\Re{\mbox{\rm\,Re\,}}
\def\Res{\mbox{\rm\,Res\,}}
\def\div{\mbox{\rm\,div\,}}
\def\wtu{\widetilde{u}}
\def\whM{\widehat{M}}
\def\whc{\widehat{c}}
\def\whf{\widehat{f}}
\def\whu{\widehat{u}}
\def\whp{\widehat{p}}
\def\wht{\widehat{t}}
\def\whU{\widehat{U}}
\def\whphi{\widehat{\phi}}
\def\whOmega{\widehat{\Omega}}

\newtheorem{theorem}{Theorem}
\newtheorem{proposition}{Proposition}
\newtheorem{definition}{Definition}
\newtheorem{lemma}{Lemma}

\newtheorem{remark}{Remark}

\def\cqfd{\hfill$\Box$}
\title{Helicoidal minimal surfaces of prescribed genus, II}
\author{David Hoffman}
\address{Department of Mathematics\\ Stanford University\\ Stanford, CA 94305}
\email{hoffman@math.stanford.edu}
\author{Martin Traizet}
\address{Laboratoire de Math\'{e}matiques et Physique Th\'{e}orique,Universit\'{e} Fran\c cois Rabelais, 37200 Tours, France}
\email{Martin.Traizet@lmpt.univ-tours.fr}
\author{Brian White}
\address{Department of Mathematics\\ Stanford University\\ Stanford, CA 94305}
\thanks{The research of the third author was supported by NSF
   grant~DMS--1105330}
\email{white@math.stanford.edu}
\date{April 22, 2013.}
\begin{document}
\maketitle

\bigskip

{\em Abstract: In this paper we prove that for each positive integer $g$, there exists a complete
minimal surface of genus $g$ that is properly embedded in three-dimensional euclidean space
and that is asymptotic to the helicoid.}

\bigskip

\section{Introduction}
In this paper we prove
\begin{theorem}
\label{theorem1}
For each positive integer $g$, there exists a complete minimal surface of genus $g$ that is properly embedded in $\R^3$ and asymptotic to the helicoid.
\end{theorem}
Let $\S^2(r)$ be the round sphere of radius $r$.
Helicoidal minimal surfaces in $\S^2(r)\times\R$ of prescribed genus have been constructed by the authors in \cite{partI}. We obtain helicoidal minimal surfaces in $\R^3$ of prescribed genus by letting the radius $r$ go to infinity.

\medskip

Our model for $\S^2(r)$ is $\C\cup\{\infty\}$ with the conformal metric obtained by stereographic projection:
\begin{equation}
\label{equation-metric-S2}
\lambda^2 |dz|^2 \quad\mbox{ with }
\lambda=\frac{2 r^2}{r^2+|z|^2},
\end{equation}
In this model, the equator is the circle $|z|=r$.
Our model for $\S^2(r)\times\R$ is $(\C\cup\{\infty\})\times\R$ with the metric
\begin{equation}
\label{equation-metric-S2xR}
\lambda^2 |dz|^2+dt^2,\quad (z,t)\in(\C\cup\{\infty\})\times\R.
\end{equation}
When $r\to\infty$, this metric converges to the euclidean metric 
$4|dz|^2+dt^2$ on $\C\times\R=\R^3$. (This metric is isometric to
the standard euclidean metric by the map $(z,t)\mapsto (2z,t)$.)

\medskip

Let $H$ be the standard helicoid in $\R^3$, defined by the equation
$$x_2\cos x_3=x_1\sin x_3.$$
It turns out that $H$ is minimal for the metric \eqref{equation-metric-S2xR} for any value of $r$, although not complete anymore (see Section 2 in \cite{partI}).
We complete it by adding the vertical line $\{\infty\}\times\R$, and still denote it $H$.
This is a complete, genus zero, minimal surface in $\S^2\times\R$.

\medskip

The helicoidal minimal surfaces in $\S^2\times\R$ constructed by the authors in \cite{partI}
have any prescribed genus. In this paper we only consider those of even genus, which have
one additional symmetry (denoted $\mu_E$ below).
Fix some positive integer $g$. For any radius $r>0$, there exist, by Theorem 1 in
\cite{partI},
two distinct helicoidal minimal surfaces of even genus $2g$ in $\S^2(r)\times\R$, which we denote
$M_+(r)$ and $M_-(r)$. Each one has two ends (corresponding to the two ends of $\S^2(r) \times\R$), each asymptotic
to $H$ or to a vertical translate of $H$.

\medskip

\begin{figure}
\label{figure1}
\begin{center}
\includegraphics[height=60mm]{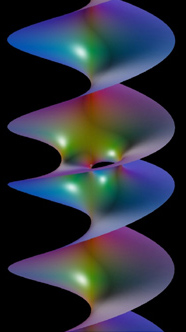}
\hspace{2cm}
\includegraphics[height=60mm]{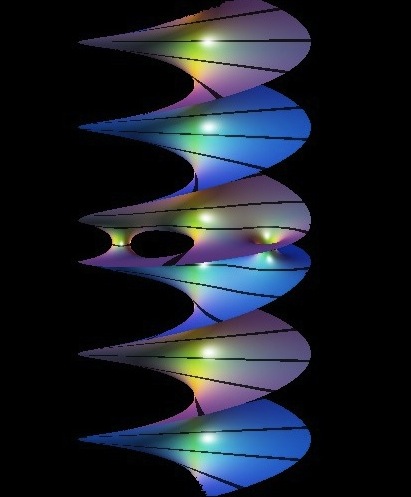}
\end{center}
\caption{
{\bf Left:} A genus-one helicoid, computed by David Hoffman, Hermann Karcher and Fusheng Wei.
{\bf Right:}  A  genus-two helicoid, computed by  Martin Traizet.
Both surfaces were computed numerically using the Weierstrass representation
and the images were made with Jim Hoffman using
visualization software he helped  to develop.
}
\end{figure}

Regarding the limit as $r\to\infty$, the following result was already proved in
\cite{partI}, Theorem 2.
\begin{theorem}
\label{theorem2}
Let $s\in\{+,-\}$. Let $R_n$ be a sequence of radii diverging to infinity.
A subsequence of the genus-$2g$ surfaces $M_s(R_n)$ converges to a minimal surface $M_s$ in $\R^3$
asymptotic to the helicoid $H$. The convergence is smooth convergence on compact
sets.
Moreover,
\begin{itemize}
\item the genus of $M_s$ is at most $g$,
\item the genus of $M_+$ is even,
\item the genus of $M_-$ is odd,
\item the number of points in $M_s\cap Y$ is $2\,\text{\rm genus}(M_s)+1$.
\end{itemize}
\end{theorem}

\medskip
The main result of this paper is the following theorem, from which Theorem \ref{theorem1}
follows.
\begin{theorem}
\label{theorem3}
If $g$ is even, then $M_+$ has genus $g$. If $g$ is odd, then $M_-$ has genus $g$.
\end{theorem}
\subsection{Pictures}
A genus-2 helicoid was computed numerically by the second author in 1993
while he was a postdoc in Amherst (see Figure \ref{figure1}, right).
Helicoids of genus up to six have been computed by  Markus Schmies \cite{schmies} using the
theoretical techniques developed by Alexander Bobenko \cite{bobenko}.
These surfaces were computed using the Weierstrass Representation and
the Period Problem was solved numerically.
There is of course no evidence that these numerically computed examples
are the same as the ones obtained in Theorem \ref{theorem1}, but they
share the same symmetries.
Pictures of these numerical examples suggest that
an even genus helicoid
looks  like a helicoid with an even number of handles far from the axis,
and  an odd genus helicoid  looks
like  a genus-one helicoid with an even number of handles far from the axis,
the spacing between the handles getting larger and larger as the genus increases.
\section{Preliminaries}
\subsection{Symmetries}
Let us recall some notation from \cite{partI}.
The real and imaginary axes in $\C$ are denoted $X$ and $Y$.
The circle $|z|=r$ is denoted $E$ (the letter $E$ stands for ``equator").
Note that $X$, $Y$ and $E$ are geodesics for the metric \eqref{equation-metric-S2}.
We identify $\S^2$ with $\S^2\times\{0\}$, so $X$, $Y$ and
$E$ are horizontal geodesics in $\S^2\times\R$.
The antipodal points $(0,0)$ and $(\infty,0)$ are denoted $O$ and $O^*$ respectively.
The vertical axes through $O$ and $O^*$ in $\S^2\times\R$ are denoted $Z$ and
$Z^*$, respectively.
If $\gamma$ is a horizontal or vertical geodesic in $\S^2\times\R$, the $180^{\circ}$
rotation around $\gamma$ is denoted $\rho_{\gamma}$. This is an isometry of
$\S^2\times\R$.
The reflection in the vertical cylinder $E\times\R$ is denoted $\mu_E$. This is an
isometry of $\S^2\times\R$. In our model,
$$\mu_E(z,t)=\left(\frac{r^2}{\overline{z}},t\right).$$

\medskip

The helicoid $H$ in $\S^2\times\R$ contains the geodesic $X$, the axes $Z$ and $Z^*$ and
meets the geodesic $Y$ orthogonally at the points $O$ and $O^*$.
 It is invariant by
$\rho_X$, $\rho_Z$, $\rho_{Z^*}$, $\mu_E$ (which reverse its orientation) and
$\rho_Y$ (which preserves it).

\medskip

The genus-$2g$ minimal surfaces $M_+(r)$ and $M_-(r)$ in $\S^2(r)\times\R$
have the following properties (see  Theorem 1 of \cite{partI}):
\begin{proposition}
\label{proposition*1}
Let $s\in\{+,-\}$. Then:
\begin{enumerate}
\item $M_s(r)$ is complete, properly embedded and has a top end and a bottom end, each asymptotic to $H$ or a vertical translate of $H$,
\item $M_s(r)\cap H=X\cup Z\cup Z^*$. In particular, $M_s(r)$ is invariant by $\rho_X$,
$\rho_Z$ and $\rho_{Z^*}$, each of which reverses its orientation.
\item $M_s(r)$ is invariant by the reflection $\mu_E$, which reverses its orientation,
\item $M_s(r)$ meets the
geodesic $Y$
orthogonally at $4g+2$ points and is invariant under $\rho_Y$, which preserves its orientation. Moreover, $(\rho_Y)_*$ acts on $H_1(M_s(r),\Z)$ by multiplication by $-1$.
\end{enumerate}
\end{proposition}

\subsection{Setup}
\label{section-setup}
Let $R_n$ be a diverging sequence of radii.
By Theorem \ref{theorem2}, a subsequence of $M_s(R_n)$ (still denoted the same) converges to a helicoidal minimal surface $M_s$ that is helicoidal at infinity.
Let $g'$ be the
genus of $M_s$.
By the last point of Theorem \ref{theorem2},
$M_s\cap Y$ has exactly $2g'+1$ points.
It follows that
$2g'+1$ points of $M_s(R_n)\cap Y$ stay at bounded distance from the origin $O$.
By $\mu_E$-symmetry, $2g'+1$ points of
$M_s(R_n)\cap Y$ stay at bounded distance from the antipodal point $O^*$.
 There remains $4(g-g')$ points in $M_s(R_n)\cap Y$ whose distance to $O$ and $O^*$ is 
 unbounded. Let
 $$N=g-g'.$$
 We shall prove
 \begin{theorem}
 \label{theorem4}
 \label{maintheorem}
 In the above setup, $N\leq 1$.
 \end{theorem}
 Theorem \ref{theorem3} is a straightforward consequence of this theorem:
 Indeed if $g$ is even and $s=+$, we know by Theorem \ref{theorem2} that
 $g'$ is even so $N=0$ and $g=g'$. If $g$ is odd and $s=-$, then $g'$ is odd
 so again $N=0$. 
 
 \begin{remark} If we let $M_s(r)\subset\S^2(r)\times\R$, $s\in\{+,-\}$ be the two helicoidal minimal surfaces of odd genus $2g+1$ constructed in \cite{partI} (instead of even genus),
 then $M_+(R_n)$ will converge subsequentially to a minimal surface
 $M_+$ of even genus $g'$ and $\mu_E(M_+(R_n))=M_-(R_n)$ will converge subsequentially to a minimal surface $M_-$ of odd  genus $g''$. Then there are $4g+2-2g'-2g''=4N$ points
 on $M_+(R_n)\cap Y$ whose distance to $O$ and $O^*$ is unbounded.
 Following our line of argument one should be able to prove that $N\leq 1$.
 This, however, does not determine $g'$ nor $g''$, so is not enough to get the existence
 of helicoidal minimal surfaces in $\R^3$ of prescribed genus. This is the main reason why
 we only consider minimal surfaces of even genus in $\S^2\times\R$.
 The $\mu_E$ symmetry that these surfaces possess does make the proof somewhat  simpler.
 The only point where we use it in a fundamental way is in the proof of Proposition \ref{proposition-case3a} where we use Alexandrov reflection. Another argument would be required at this point in the odd-genus case.
 \end{remark}
 To prove Theorem \ref{maintheorem},
 assume that $N\geq 1$. We want to prove that $N=1$ by studying the $4N$ points whose
 distance to $O$ and $O^*$ is unbounded.
To do this, it is necessary to work on a different scale. Let $R_n$ be a sequence of radii with $R_n\to\infty$. Define
 $$M_n=\frac{1}{R_n} M_s(R_n)\subset \S^2(1)\times\R.$$
 This is a minimal surface in $\S^2(1)\times\R$.
 Each end of $M_n$ is asymptotic to a vertical translate of a helicoid of pitch
$$t_n=\frac{2\pi}{R_n}.$$
(The pitch of a helicoid with counterclockwise rotation is twice the distance between consecutive sheets.
The standard helicoid has pitch $2\pi$.)
Observe that $t_n\to 0$.
By the definition of $N$, the intersection $M_n\cap Y$ has $4N$ points whose distance to $O$ and $O^*$ is $\gg t_n$.
 Because $M_n$ is symmetric with respect to $180^{\circ}$ rotation $\rho_X$ around $X$, there are $2N$ points on the positive $Y$-axis.
 We order these by increasing
 imaginary part:
 $$p'_{1,n},p''_{1,n},p'_{2,n},p''_{2,n},\cdots,p'_{N,n},p''_{N,n}.$$
 Because of the $\rho_X$-symmetry,
 the $2N$ points on the negative $Y$-axis are the conjugates of these points.
 Define $p_{j,n}$ to be the midpoint of the interval $[p'_{j,n},p''_{j,n}]$ and
 $r_{j,n}$ to be half the distance in the spherical metric from $p'_{j,n}$ to
 $p''_{j,n}$.
   We have
  $$0<\Im p_{1,n}<\Im p_{2,n}<\cdots<\Im p_{N,n}.$$
By $\mu_E$-symmetry, which corresponds to inversion in the unit circle,
  \begin{equation}
  \label{equation**}
  p_{N+1-i,n}=\frac{1}{\overline{p_{i,n}}}.
  \end{equation}
  In particular, in case $N$ is odd, $p_{\frac{N+1}{2},n}=i$.
  
  \medskip
  
  For $\lambda>1$ sufficiently large,  let  $\mathcal Z_n(\lambda)$ be the part of $M_n$ lying inside of the vertical cylinders  of radius $\lambda t_n$ around $Z$ and $Z^*$:
 \begin{equation}\mathcal Z_n(\lambda)=\{q=(z,t)\in M_n\,:\, d(Z\cup Z^*, q)<\lambda t_n\}.
\end{equation}
Also define
 $D_{j,n}(\lambda) =\{z\,:\, d(z, p_{j,n})<\lambda r_{j,n}\}$. Consider the intersection of $M_n$ with the vertical cylinder over $D_{j,n}(\lambda)$, and let $C_{
j,n}(\lambda)$ denote the component of this intersection that contains
 the points $\{p'_{j,n}, p''_{j,n}\}$.
Define
\begin{equation}\label{All_Necks_1}
\mathcal C_n(\lambda)=\bigcup_{j=1}^N C_{j,n}(\lambda)\cup \overline{C_{j,n}(\lambda)}.
\end{equation}
The following proposition is key to setting up the analysis we will do in this paper to show that at most one handle is lost in taking the limit as $R_n\rightarrow\infty$. In broad terms, it says that  near the points $p_{j,n}$, catenoidal necks are forming on a small scale, and after removing these necks and a neighborhood of the axes,  what is left is a pair of symmetric surfaces which are vertical graphs over a half-helicoid.

 \begin{proposition}\label{Setup} Let  $N=g-g'$, $t_n$,  and  $M_n \subset \S^2(
1)\times \R$,  be as above.
Then
\begin {enumerate}        \item[i.]  For each $j$, $1 \leq j \leq N$,  the surface $\frac{1}{r_{j,n}}(M_n-p_{j,n})$ converges to  the standard  catenoid $\mathbf C$ with vertical
 axis  and waist circle of radius $1$ in $\R^3$. In particular, the distance (in the spherical metric) $d(p_{j,n},p_{j+1,n})$ is $\gg r_{j,n}$. Moreover, $t_n\gg r_{j,n}$ and the $C_{j,n}(\lambda)$ are close to catenoidal necks with collapsing radii.        \item[ii.]  Given $\epsilon>0$, there exists a $\lambda >0$ such that
  $$M_n'=M_n\setminus (\mathcal Z_n(\lambda)\cup\mathcal C_n(\lambda))$$
  has the following properties
                 \begin{enumerate}
                         \item The slope of the tangent plane at any point of $M'_n$ is less than $\epsilon$.
                         \item $M'_n$ consists of two components related by
 the symmetry $\rho _Y$, rotation by $180^\circ$ around $Y$.
                        \item $M'_n$ intersects $t_nH$ in a subset of the axis $X$ and nowhere else, with one of its components intersecting in a ray of the positive $X$-axis, the other in a ray of $X^-$. Each component is graphical over its projection onto the half-helicoid (a component of $t_nH\setminus (Z\cup Z^*)$) that it intersects.
                                        \end{enumerate}

\end{enumerate}
 \end{proposition}
 This proposition is proved in Theorem 16.9 and Corollary 16.13 of \cite{partI}.
 The notations in \cite{partI} are slightly different: See Remark \ref{remark-notations}
 below.
 \medskip
 
Passing to a subsequence, $p_j=\lim p_{j,n}\in i\R^+\cup\{\infty\}$
exists for all $j\in[1,N]$. We have $p_1\in[0,i]$, and we will consider the following three cases:
\begin{equation}
\label{equation*2}
\begin{array}{ll}
\bullet & \text{Case 1: } p_1\in(0,i),\\
\bullet & \text{Case 2: } p_1=0,\\
\bullet & \text{Case 3: }  p_1=i.
\end{array}
\end{equation}
We will see that Case 1 and Case 2 are impossible, and that $N=1$ in Case 3.
\subsection{The physics behind the proof of Theorem \ref{theorem4}}
\label{section-physics}
Theorem \ref{theorem4} is proved by evaluating the surface tension in the
$Y$-direction on each catenoidal neck.
Mathematically speaking, this means the flux of the horizontal Killing field tangent to the 
$Y$-circle in $\S^2\times\R$.
On one hand, this flux vanishes at each neck by $\rho_Y$-symmetry
(see Lemma \ref{lemma*8}).
On the other hand, we can compute the limit $F_i$ of the surface tension on the $i$-th catenoidal neck (corresponding to $p_i=\lim p_{i,n}$) as $n\to\infty$, after suitable scaling.

\medskip

Assume for simplicity that the points $O$, $p_1,\cdots,p_N$ and $O^*$ are
distinct. Recall that the points $p_1,\cdots p_N$ are on the positive imaginary $Y$-axis.
For $1\leq j\leq N$, let $p_j=i y_j$, with $0<y_j<\infty$.
Then we will compute that
$$F_i=c_i^2\frac{1-y_i^2}{1+y_i^2}+\sum_{j=1\atop j\neq i}^N c_i c_j f(y_i,y_j)$$
where the numbers $c_i$ are positive and proportional to the size of
the catenoidal necks and
$$f(x,y)=\frac{-2\pi^2}{(\log x-\log y)|\log x-\log y+i\pi|^2}.$$
Observe that $f$ is antisymmetric and $f(x,y)>0$ when $0<x<y$.
We can think of the point $p_i$ as a particle with mass $c_i$ and interpret $F_i$ as a force of gravitation type.
The particles $p_1,\cdots,p_N$ are attracted to each other
and we can interpret the first term by saying that each particle
$p_i$ is repelled from the fixed antipodal points $O$ and $O^*$.
All forces $F_i$ must vanish. It is physically clear that no equilibrium is possible unless
$N=1$ and $p_1=i$. Indeed in any other case, $F_1>0$.

\medskip

This strategy is similar to the one followed in \cite{traizet1} and \cite{traizet2}.
The main technical difficulty is that we cannot guarantee that the points $O$, $p_1,\cdots,
p_N$ and $O^*$ are distinct. The distinction between Cases 1, 2 and 3 in \eqref{equation*2} stems from this problem.
\def\CC{\widetilde{\C^*}}
\subsection{The space $\CC$}
To compute forces we need to express $M_n$ as a graph.
For this, we need to express the helicoid itself as a graph, away from its axes $Z$ and
$Z^*$.
Let $\CC$ be the universal cover of $\C^*$. Of course, one can identify $\CC$ with $\C$ by mean of the exponential function. It will be more convenient to see $\CC$ as the covering space obtained by analytical continuation of $\log z$, so each point of $\CC$ is a point of $\C^*$ together with a determination of its argument : points are couples $(z,\arg(z))$,
although in general we just write $z$.
The following two involutions of $\CC$ will be of interest:
\begin{itemize}
\item $(z,\arg(z))\mapsto (\overline{z},-\arg(z))$, which we write simply as $z\mapsto \overline{z}$. The fixed points are $\arg z=0$.
\item $(z,\arg(z))\mapsto (1/\overline{z},\arg(z))$, which we write simply as
$z\mapsto 1/\overline{z}$. The fixed points are $|z|=1$.
\end{itemize}
The graph of the function
$\frac{t}{2\pi}\arg z$ on $\CC$ is one half of a helicoid of pitch $t$.
\subsection{The domain $\Omega_n$ and the functions $f_n$ and $u_n$}
\label{section25}
  By Proposition~\ref{Setup}, away from the axes $Z\cup Z^*$ and the points
  $p_{j,n}$, we  may consider $M
_n$ to be the union of two  multigraphs. We wish to express  this part of $M_n$ 
as a  pair of graphs over a subdomain of $\CC$.
We will allow ourselves the freedom to write  $z$ for a point $(z,\arg z)\in \CC$ when its argument is clear from the context. Thus we will write $p_{j,
n}$ for the point $(p_{j,n}, \pi/2)$ in $\CC$ corresponding to the points on $M_n\cap Y
$ in Proposition~\ref{Setup}.  Define
\begin{equation}\label{D_n_Lambda}
D_n(\lambda)=\{\,(z,\arg z)\,:\, |z|<\lambda t_n \text{ or } |z|>\frac{1}{\lambda t_n}\},
\end{equation}
\begin{equation}\label{D_j_n}
D_{j,n}(\lambda)=\{\,(z,\arg z)\,:\, d(p_{j,n}, z)<\lambda r_{j,n}\text{ and } 0<\arg z <\pi\}
\end{equation}
and
\begin{equation}\label{Domain_1}
\Omega_n=\Omega_n(\lambda)=\CC\setminus \left(D_n(\lambda) \cup \bigcup_{j=1}^{N}D_{j,n}(\lambda)\cup\overline{D_{j,n}(\lambda)}\right ).
\end{equation}
According to Statement~$ii.$ of Proposition~\ref{Setup},  there exists a $\lambda>0$
such that for sufficiently large $n$,   $$M'_n=M_n\cap (\Omega_n(\lambda)\times\R)$$  is the union of two graphs related by $\rho_Y$-symmetry, and each graph intersects  the helicoid of pitch $t_n$ in a subset of the  $X$-axis. Only one of these graphs can contain points on the positive $X$-axis. We choose this component and write it
 as the graph of a function $f_n$ on the domain $\Omega_n$. We may write
 \begin{equation}\label{equation*3}
 f_n(z)=\frac{t_n}{2\pi}\arg z-u_n(z).
 \end{equation}
The function $u_n$ has the following properties:
\begin{equation}
\label{equation*4}
\begin{array}{ll}
\bullet & \text{ $u_n(\overline z) =-u_n(z)$. In particular, $u_n=0 $ on $\arg z=0$.}\\
\bullet & \text{$u_n(1/\overline z) =u_n(z)$ In particular, $\partial u_n/\partial \nu =0$ on $|z|=1$.}\\
\bullet & \text{ $0<u_n<t_n/2$ when $\arg z >0$.}
\end{array}
\end{equation}
The first two assertions follow from the symmetries of $M_n$. See Proposition~\ref{proposition*1} (Statements 2 and 3),  and the discussion preceding it.  The third assertion follows Proposition~\ref{Setup}, Statement ~$ii.c$, which implies that
$$0<|u_n|<t_n/2$$
  when $\arg z>0$, since the vertical distance between the sheets of $t_nH$ is equal to $t_n/2$.
Now choose a point $z_0$ in the domain of $f_n$  that is near the a point $p_{j,n}$. Then $|f_n(z_0)|$ is small,
and $\arg z_0$ is near $\pi/2$. Hence $f_n(z_0)\sim t_n/4 -u_n(z_0)$, which implies that $u_n(z_0)>0$. We conclude that
$0<u_n<t_n/2$ when $\arg z >0$, as claimed.

\begin{figure}
\label{figure-domain}
\begin{center}
\includegraphics[width=60mm]{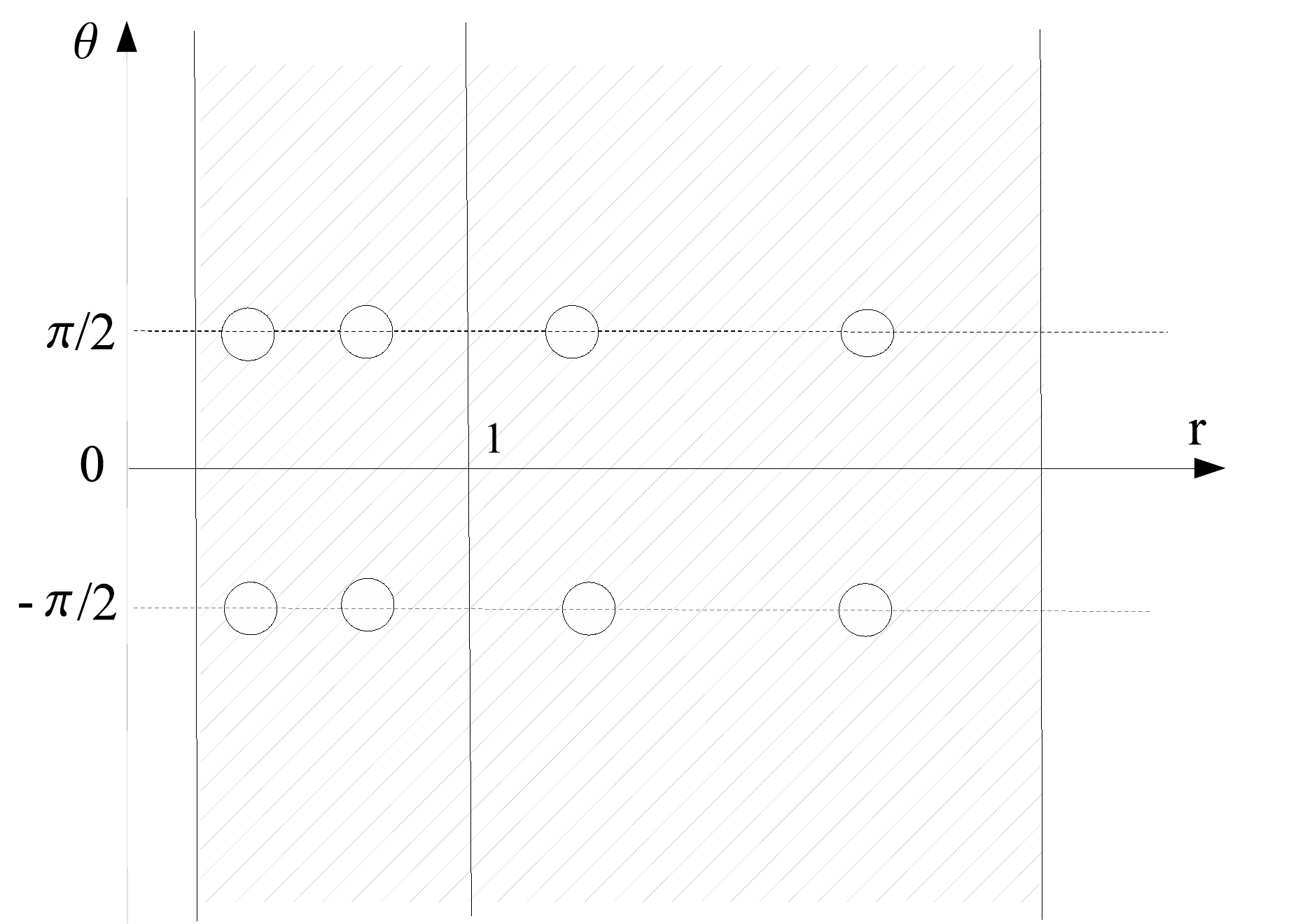}
\end{center}
\caption{The domain $\Omega_n$ in polar coordinates, $z=re^{i\theta}$.
The function $u_n$ is positive for $\theta>0$. 
 The line $r=1$ corresponds to the unit circle $|z|=1$. The white strip on the left corresponds to the projection of the vertical cylinder of radius $\lambda t_n$ about the $Z$-axis, and the region to the right of the shaded domain is its image by the inversion through the unit circle. The small disks correspond to the vertical cylinders of radius $\lambda r_{j,n}$ (in the spherical metric).
}
\end{figure}
\begin{remark}
\label{remark-notations} There are some notational differences between \cite{partI} and the present paper:
\begin{itemize}
\item In \cite{partI}, $z$ denotes the third coordinate in $\R^3$. Here $z=x+i y$ is a complex variable.
\item 
In \cite{partI}, the pitch of a helicoid is denoted $2\eta$.
 Here it is denoted $t$.
\item 
In Section 16 of \cite{partI},
 the angle $\theta$ is measured from the positive
$Y$-axis, whereas here, it is measured from the positive $X$-axis.
\item The conformal factor for the spherical metric in \cite{partI} is slightly different from
\eqref{equation-metric-S2}. It is chosen there so that it converges to $1$ as $r\to\infty$.
\item In paper \cite{partI}, $M_s(r)$ has genus $g$, whereas here it has genus $2g$.
\end{itemize}
\end{remark}
\subsection{Organization of the paper}
We deal with Cases 1, 2 and 3, as listed in \eqref{equation*2}, separately. In each case, we first state, without proof, a proposition
which describes the asymptotic behavior of the function $u_n$ defined by \eqref{equation*3} as $n\to\infty$.
We use this result to compute forces and obtain the required result (namely, $N=1$
or a contradiction). Then, we prove the proposition.
Finally, an Appendix contains analytic and geometric results relevant to minimal surfaces in $\S^2\times\R$,
which are used in this paper.
\section{Case 1: $p_1\in (0,i)$}
\label{section2}
For $p\in\CC$, let $h_p$ be the harmonic function defined on
$\CC\setminus\{p,\overline{p}\}$ by
\begin{equation}
\label{equation*5}
h_p(z)=-\log\left|\frac{\log z-\log p}
{\log z-\log\overline{p}}\right|.
\end{equation}
Note that since $p$ and $z$ are in $\CC$, both come with a determination of their logarithm, so the function $h_p$ is well defined.
This function has the same symmetries as $u_n$:
\begin{equation}
\label{equation*6}
\begin{array}{ll}
\bullet & h_p(\overline{z})=-h_p(z),\\
\bullet & h_p(1/\overline{z})=h_{1/\overline{p}}(z).\\
\bullet &\text{ Moreover, if $\arg p$ and $\arg z$ are positive then $h_p(z)>0$.}
\end{array}
\end{equation}
\begin{remark}
The function $(z,p)\mapsto -h_p(z)$ is the Green function of the domain $\arg z>0$ of $\CC$.
\end{remark}
Recall that $p_i=\lim p_{i,n}$. It might happen that several points $p_j$ are equal to $p_i$.
In this case, we say that we have a cluster at $p_i$.
Let $m$ be the number of distinct points amongst $p_1,\cdots,p_N$.
Relabel the points so that $p_1,\cdots,p_m$ are distinct and
$$\Im p_1<\Im p_2\cdots<\Im p_m.$$
(In other words, we have selected one point in each cluster.)
Let us define
\begin{equation}
\label{equation*7}
\wtu_n=
\frac{|\log t_n|}{t_n} u_n.
\end{equation}
\begin{proposition}
\label{proposition1}
Assume that $p_1\neq 0$. Then there exists a subsequence and non-negative real numbers
$c_0,\cdots,c_m$ such that
\begin{equation}
\label{equation-limit-un}
\wtu(z):=\lim \wtu_n(z)=c_0\arg z+\sum_{i=1}^m
c_ih_{p_i}(z).
\end{equation}
The convergence is the usual smooth uniform convergence on compact subsets of $\CC$ minus the points
$p_i$, $-p_i$ for $1\leq i\leq m$.
Moreover, for $1\leq i\leq m$,
\begin{equation}
\label{equation-ci}
c_i=\lim\frac{|\log t_n|}{t_n}\frac{\phi_{i,n}}{2\pi}
\end{equation}
where $\phi_{i,n}$ is the vertical flux of $M_n$ on
the graph of $f_n$ restricted to the circle $C(p_i,\varepsilon)$ for a fixed, small enough $\varepsilon$.
\end{proposition}
In other words, $\phi_{i,n}$ is the sum of the vertical fluxes on the catenoidal necks
corresponding to the points $p_{j,n}$ such that $p_j=p_i$.
\begin{remark}
We allow $p_1=i$ as this proposition will be used in Case 3, Section \ref{section4}.
\end{remark}
This proposition is proved in Section \ref{section-proof-proposition1} by estimating the Laplacian of $u_n$ and
constructing an explicit barrier, from which we deduce that a subsequence converges to a limit harmonic function on $\CC$ with logarithmic singularities at $\pm p_1,\cdots,\pm p_m$.
\begin{remark}
In Proposition \ref{proposition1},  it is easy to show using Harnack's inequality
that we can choose numbers
$\lambda_n>0$ so that $\lambda_n u_n$ converges subsequentially to a nonzero limit of the form \eqref{equation-limit-un}.
(One fixes a point $z_0$ and lets $\lambda_n= 1/u_n(z_0)$.)   However, for us it is crucial 
that we can choose $\lambda_n$ to be $\frac{|\log t_n|}{t_n}$; it means that in later calculations, we will be able to ignore terms that are $o(\frac{|\log t_n|}{t_n})$.
\end{remark}

For all we know at this point, the limit $\wtu$ might be zero.
We will prove this is not the case:
\begin{proposition}
\label{proposition2} For each $i\in[1,m]$, $c_i>0$.
\end{proposition}
This proposition is proved  in Section \ref{section22} using a height estimate, Proposition \ref{proposition-height}, to estimate the vertical flux of the catenoidal necks.

\medskip
From now on assume that $p_1\in (0,i)$.
Fix some small number $\varepsilon$ and let $F_n$ be the flux of the Killing field
$\chi_Y$ on the circle $C(p_1,\varepsilon)$.
The field $\chi_Y$ is the Killing field associated with rotations with respect to poles
whose equator is the $Y$-circle
(see Proposition \ref{proposition-flux1} in Appendix \ref{appendix-flux}).
On one hand, we have:
\begin{lemma}
\label{lemma*8}
$F_n=0$.
\end{lemma}
\begin{proof}
Let $C_n$ be the graph of $f_n$ restricted to the circle $C(p_1,\varepsilon)$.
By Proposition \ref{proposition*1}, statement (4),
$C_n$ together with its image $\rho_Y (C_n)$
bound a compact region in $M_n$.  Thus the flux of the Killing field $\chi_Y$ on 
$C_n \cup \rho_Y (C_n)$ is $0$.   By $\rho_Y$-symmetry,
this flux is twice the flux $F_n$ of $\chi_Y$ on $C_n$.
Thus $F_n=0$.
\end{proof}

\medskip

On the other
hand, $F_n$ can be computed using Proposition \ref{proposition-flux2} from Appendix
\ref{appendix-flux}:
\begin{align}F_n&=-\Im\int_{C(p_1,\varepsilon)}
2\left(\frac{\partial}{\partial z}\left(\frac{t_n}{2\pi}\arg z-u_n\right)\right)^2\frac{i}{2}(1-z^2)dz+O(t_n^4)\nonumber\\
&=-\Re\int_{C(p_1,\varepsilon)}
\left(\frac{t_n}{4\pi i z} - u_{n,z}\right)^2 (1-z^2) dz+O(t_n^4)\nonumber\\
&=-\Re\int_{C(p_1,\varepsilon)} \left( \frac{-t_n^2}{16\pi^2 z^2}
-\frac{2 t_n}{4\pi i z} u_{n,z}+(u_{n,z})^2\right)(1-z^2)dz+O(t_n^4)\nonumber\\
\label{eq-integral}
&=\Re\int_{C(p_1,\varepsilon)} \left( 
\frac{2 t_n}{4\pi i z} u_{n,z}-(u_{n,z})^2\right)(1-z^2)dz+O(t_n^4)
\end{align}
The second equation comes from $\frac{\partial}{\partial z}\arg z=\frac{1}{2iz}$.
The fourth equation is a consequence of the fact that $\frac{1-z^2}{z^2}$ has no residue at
$p_1\neq 0$.
The first term in \eqref{eq-integral} (the cross-product) is a priori the leading term.
However we can prove that this term can be neglected:
\begin{proposition}
\label{proposition3}
\begin{equation}
\label{equation*9}
\lim\left(\frac{\log t_n}{t_n}\right)^2 F_n=-\Re\int_{C(p_1,\varepsilon)} (\wtu_z)^2(1-z^2)dz
\end{equation}
where $\wtu$ is defined in  \eqref{equation-limit-un} as the limit of $\frac{|\log t_n|}{t_n}u_n$.
\end{proposition}
This proposition is proved in Section \ref{section23} using a Laurent series expansion to estimate the first term in \eqref{eq-integral}.

\medskip

Assuming these results, we now prove
\begin{proposition}
\label{proposition-case1}
Case 1 is impossible.
\end{proposition}
Proof:
According to Lemma \ref{lemma*8}, the flux $F_n$ is zero.
Hence the limit in \eqref{equation*9} is zero. We compute that limit and show that it is nonzero.
\medskip

Differentiating equation \eqref{equation-limit-un}, we get
$$\wtu_z=
\frac{c_0}{2iz}-\sum_{i=1}^m\frac{c_i}{2z}\left(\frac{1}{\log z-\log p_i}-\frac{1}{\log z-\log\overline{p_i}}\right).$$
Therefore,
\begin{eqnarray*}
\lefteqn{\Res_{p_1}(\wtu_z)^2(1-z^2)} \\
&=&
\Res_{p_1}\frac{1-z^2}{4z^2}\left[\frac{c_1^2}{(\log z-\log p_1)^2}
+2\frac{c_1}{\log z-\log p_1}
\right.\\
&&\left.\left(\frac{-c_0}{i}-\frac{c_1}{\log z-\log \overline{p_1}}
+\sum_{i=2}^m \frac{c_i}{\log z-\log p_i}-\frac{c_i}{\log z-\log\overline{p_i}}\right)\right]\\
&=& -\frac{c_1^2(1+p_1^2)}{4p_1}+\frac{c_1(1-p_1^2)}{2p_1}\left(
\frac{-c_0}{i}-\frac{c_1}{\log p_1-\log \overline{p_1}}\right.\\
&&\left.+\sum_{i=2}^m\frac{c_i}{\log p_1-\log p_i}-\frac{c_i}{\log p_1-\log \overline{p_i}}\right).
\end{eqnarray*}
(See Proposition \ref{proposition-residue} in Appendix \ref{appendix-residue}
for the residue computations.)
Write $p_j=i y_j$ for $1\leq j\leq m$ so all $y_j$ are positive numbers.
By Lemma \ref{lemma*8}, equation \eqref{equation*9} and the Residue Theorem,

\begin{eqnarray}
0&=&-\Re\int_{C(p_1,\varepsilon)}(\wtu_z)^2(1-z^2)dz\nonumber
\\
&=&-\Re\left[ 2\pi i \frac{y_1^2+1}{4iy_1}\left(c_1^2\frac{y_1^2-1}{y_1^2+1}
+2 c_1 \left(-\frac{c_0}{i}-\frac{c_1}{i\pi}
\right. \right. \right.\nonumber\\
& &\left. \left. \left.
+\sum_{i=2}^m\frac{c_i}{\log y_1-\log y_i}-\frac{c_i}{\log y_1-\log y_i+i\pi}
\right)\right)\right]\nonumber\\
\label{equation*10}
&=&\frac{\pi(y_1^2+1)}{2y_1}\left[
c_1^2\frac{1-y_1^2}{y_1^2+1}+\sum_{i=2}^m \frac{-2\pi^2\,c_1 c_i}{(\log y_1-\log y_i)|\log y_1-\log y_i+i\pi|^2}\right].
\end{eqnarray}
Now $y_1<1$ and $y_1<y_i$ for all $i\geq 2$, so all terms in \eqref{equation*10} are positive.
This contradiction proves Proposition \ref{proposition-case1}.\cqfd
\begin{remark}
The bracketed term in \eqref{equation*10} is precisely the expression for the force $F_1$ in
Section \ref{section-physics}.
\end{remark}
\subsection{Barriers}
\label{section21}
In this section we introduce various barriers that will be used to prove Proposition
\ref{proposition1}.
Fix some $\alpha\in(0,1)$. 
\begin{definition}
\label{definition*11}
$A_n$ is the set of points
$(z,\arg z)$ in $\CC$ which satisfy $t_n^{\alpha}<|z|<1$ and $\arg z>0$,
minus the disks $D(p_{i,n},t_n^{\alpha})$ for $1\leq i\leq N$.
\end{definition}
It is clear that $A_n\subset\Omega_n$ for large $n$, since $t_n^{\alpha}\gg t_n$.
Moreover, if $z\in A_n$ then $d(z,\partial\Omega_n)\geq t_n^{\alpha}/2$.

\begin{remark} We work in the hemisphere $|z|\leq 1$ where the conformal factor
of the spherical metric in \eqref{equation-metric-S2}
satisfies $1\leq\lambda\leq 2$. Hence euclidean and spherical
distances are comparable. We will use euclidean distance. Also the euclidean and
spherical Laplacians are comparable. The symbol $\Delta$ will mean euclidean Laplacian.

\medskip

By the disk $D(p,r)$ in $\CC$ (for small $r$) we mean the points
$(z,\arg z)$ such that $|z-p|<r$ and $\arg z$ is close to $\arg p$.
\end{remark}

Let $\delta$ be the function on $A_n$ defined by
$$\delta(z)=\left\{\begin{array}{ll}
\min\{|z|,|z-p_{1,n}|,\cdots,
|z-p_{N,n}| \}
 &\mbox{ if }0<\arg z<\pi\\
 |z| &\mbox{ if }\arg(z)\geq\pi
 \end{array}\right.$$

\begin{lemma} 
\label{lemma1}
There exists a constant $C_1$ such that
in the domain $A_n$, the function $u_n$ satisfies
$$|\Delta u_n|\leq C_1\frac{t_n^3}{\delta^4}$$.
\end{lemma}
Proof: the function $f_n(z)=\frac{t_n}{2\pi}\arg z-u_n(z)$ satisfies the minimal surface
equation, and $|\Delta f_n|=|\Delta u_n|$.
The proposition then follows from Proposition \ref{proposition-schauder} in Appendix \ref{appendix-schauder} (a straightforward application of Schauder estimate). More precisely:
\begin{itemize}
\item if $0<\arg z<\pi$, we apply
Proposition \ref{proposition-schauder} on the domain
$$A'_n=\{w\in\Omega_n\,:\,-\pi/2<\arg w<3\pi/2, |w|<2\}.$$
The distance $d(z,\partial A'_n)$ is  comparable
to $\delta(z)$. The function $f_n$ is bounded
by $3t_n/4$.
\item If $k\pi\leq\arg z<k\pi+\pi$ for some $k\geq 1$, we apply Proposition \ref{proposition-schauder} to the function $f_n-\frac{k}{2}t_n$ and the domain 
$$A'_n=\{w\in\Omega_n\,:\, k\pi-\pi/2<\arg w<k\pi+3\pi/2,|w|<2\}.$$
 The distance $d(z,\partial A'_n)$
 is comparable to $|z|$. The function $f_n-\frac{k}{2}t_n$ is again bounded by
$3t_n/4$.
\end{itemize}
\cqfd

Next, we need to construct a function whose Laplacian is greater than $1/\delta^4$,
in order
to compensate for the fact that $u_n$ is not quite harmonic.
Let $\chi:\R^+\to[0,1]$ be a fixed, smooth function such that $\chi\equiv 1$ on
$[0,\pi]$ and $\chi\equiv 0$ on $[2\pi,\infty)$.
\begin{lemma}
\label{lemma2}
There exists a constant $C_2\geq 1$ such that the function $g_n$ defined on $A_n$ by
$$g_n(z)=\frac{C_2}{|z|^2}+\chi(\arg z)
\sum_{i=1}^N
\frac{1}{|z-p_{i,n}|^2}$$
satisfies
\begin{equation}
\label{equation*12}
\Delta g_n\geq \frac{4}{\delta^4}.
\end{equation}
Moreover $\partial g_n/\partial \nu\leq 0$ on $|z|=1$ and
\begin{equation}
\label{equation*13}
g_n\leq \frac{C_2+
N
}{t_n^{2\alpha}}\quad\mbox{ in } A_n.
\end{equation}
\end{lemma}
Proof:
The inequality \eqref{equation*13} follows immediately from the definitions of $g_n$
and $A_n$.
The function $f$ defined in polar coordinate by $f(r,\theta)=1/r^2$ satisfies
$$|\nabla f|=\frac{2}{r^3},\qquad \Delta f=\frac{4}{r^4}.$$
Hence for $\arg z\geq 2\pi$, \eqref{equation*12} is satisfied for any $C_2\geq 1$.
Suppose $0<\arg z <\pi$. Then
$$\Delta g_n=\frac{4C_2}{|z|^4}+
\sum_{i=1}^N
\frac{4}{|z-p_{i,n}|^4}\geq \frac{4}{\delta^4}$$
so again, \eqref{equation*12} is satisfied for any $C_2\geq 1$.
If $\theta=\arg z\in [\pi,2\pi]$, we have $|z-p_{i,n}|\geq |z|=r$ and
$$|\nabla \chi(\arg z)|\leq \frac{C}{r},\qquad |\Delta\chi(\arg z)|\leq \frac{C}{r^2}.$$
Hence
$$\left|\Delta \frac{\chi(\arg z)}{|z-p_{i,n}|^2}\right|\leq
\frac{C}{r^2}\frac{1}{r^2}+2\frac{C}{r}\frac{2}{r^3}+\frac{4}{r^4}.$$
Therefore, $\Delta g_n\geq 4/r^4$ provided $C_2$ is large enough.
(The constant $C_2$ only depends on $N$ and a bound on $\chi'$ and $\chi''$.)
This completes the proof of \eqref{equation*12}.
\cqfd
\medskip

We need a harmonic function on $\CC$ that is greater than $|\log t|$ on
$|z|=t$. A good candidate would be $-\log |z|$. However this function has the wrong 
Neumann data on the unit circle. We propose the following:
\begin{lemma}
\label{lemma3}
For $0<t<1$, the harmonic function $H_t(z)$ defined for $z\in\CC$, $\arg z>0$ by
$$H_t(z)=\Im\left(\frac{\log t\log z}{\log t +i\log z}\right)$$
has the following properties :
\begin{enumerate}
\item $H_t(z)>0$ if $\arg z>0$,
\item $H_t(1/\overline{z})=H_t(z)$, hence $\partial H_t/\partial\nu=0$ on $|z|=1$,
\item $H_t(z)\geq |\log t|/2$ if $|z|=t$,
\item for fixed $t$, $H_t(z)\geq |\log t|/2$ when $\arg z\to\infty$, uniformly with respect to $|z|$
in $t\leq |z|\leq 1$,
\item for fixed $z$, $H_t(z)\to \arg z$ when $t\to 0$,
\item $H_t(z)\leq |\log z|$ if $\arg z>0$.
\end{enumerate}
\end{lemma}
Proof : it suffices to compute $H_t(z)$ in polar coordinates
$z=r e^{i\theta}$ :
$$H_t(z)=\frac{(\log t)^2\theta+|\log t|((\log r)^2+\theta^2)}{(\log t-\theta)^2+(\log r)^2}.$$
The first two points follow.
If $r=t$ then
$$H_t(z)=\frac{|\log t|}{2}\left(1+\frac{\theta^2}{2(\log t)^2+2|\log t|\theta +\theta^2}\right)
\geq \frac{|\log t|}{2}$$
which proves point 3. If $t\leq r\leq 1$ then
$$H_t(z)\geq \frac{ (\log t)^2\theta+|\log t|\theta^2}{(\log t-\theta)^2+(\log t)^2}$$
which gives point 4. Point 5 is elementary. For the last point, write
$$\left|\frac{\log t\log z}{\log t+i\log z}\right|\leq \left|\frac{\log t\log z}{\log t}\right|=|\log z|.$$
\cqfd
\medskip

\subsection{Proof of Proposition \ref{proposition1}}
\label{section-proof-proposition1}
The function $\wtu_n$ defined in \eqref{equation*7} has the following properties in $A_n$:
\begin{equation}
\label{equation*14}
\begin{array}{ll}
\bullet & \displaystyle|\Delta \wtu_n|\leq  C_1 \frac{t_n^2|\log t_n|}{\delta^4}
\text{ by Lemma \ref{lemma1}},\\
\bullet & \wtu_n\leq |\log t_n|/2,\\
\bullet & \wtu_n=0 \text{ on $\arg z=0$,}\\
\bullet & \partial \wtu_n/\partial\nu=0 \text{ on $|z|=1$.}
\end{array}
\end{equation}
The last three properties follow from \eqref{equation*4} and the fact that $A_n\subset\Omega_n$.
Consider the barrier $v_n=v_{1,n}+v_{2,n}+v_{3,n}$ where
$$v_{1,n}(z)= - C_1 t_n^2|\log t_n| g_n(z)
+C_1(C_2+
N
)t_n^{2-2\alpha}|\log t_n|,$$
$$v_{2,n}(z)=\frac{1}{\alpha}\sum_{i=1}^N
h_{p_{i,n}}(z)
=-\frac{1}{\alpha}\sum_{i=1}^N\log\left|\frac{\log z-\log p_{i,n}}
{\log z-\log\overline{p_{i,n}}}\right|,$$
$$v_{3,n}(z)=\frac{1}{\alpha}H_{t_n^{\alpha}}(z).$$
The function $v_{1,n}$ is positive in $A_n$ by the estimate \eqref{equation*13} of
Lemma \ref{lemma2}.
Observe that the second term in the expression for $v_{1,n}$ tends to $0$ as $n\to\infty$ since $\alpha<1$.
The functions $v_{2,n}$ and $v_{3,n}$ are harmonic and positive in $A_n$
(see point (1) of Lemma \ref{lemma3} for $v_{3,n}$).
\medskip

By \eqref{equation*6} and the symmetry of the set $\{p_{1,n},\cdots,p_{N,n}\}$
(see \eqref{equation**}), 
the function $v_{2,n}$ satisfies $v_{2,n}(1/\overline{z})=v_{2,n}(z)$.
Hence $\partial v_{2,n}/\partial\nu=0$ on the unit circle.
By point (2) of Lemma \ref{lemma3}, $\partial v_{3,n}/\partial\nu=0$ on the unit circle.
Therefore by Lemma \ref{lemma2},
$$\frac{\partial v_n}{\partial\nu}=\frac{\partial v_{1,n}}{\partial\nu}\geq 0
\text{ on $|z|=1$.}$$
Because $p_{i,n}\to p_i\neq 0$, we have on the circle $C(p_{i,n},t_n^{\alpha})$
$$\log |\log z - \log p_{i,n}|\simeq \log |z-p_{i,n}|$$
Hence for large $n$ and for $1\leq i \leq N$
$$v_{2,n}\geq \frac{1}{2\alpha}|\log t_n^{\alpha}|\geq \wtu_n\quad \mbox{ on } C(p_{i,n},t_n^{\alpha}).$$
Using point (3) of Lemma \ref{lemma3} and the second statement of \eqref{equation*14},
we have
$v_{3,n}\geq \wtu_n$ on the boundary component
$|z|=t_n^{\alpha}$.
So we have
\begin{equation}
\label{equation*15}
\begin{array}{ll}
\bullet &\Delta \wtu_n\geq \Delta v_n \text{ in $A_n$},\\
\bullet &\wtu_n\leq v_n \text{ on the boundaries $\arg z=0$, $|z|=t_n^{\alpha}$ and
$C(p_{i,n},t_n^{\alpha})$,}\\
\bullet &\partial \wtu_n/\partial \nu\leq \partial v_n/\partial \nu\text{ on the boundary $|z|=1$},\\
\bullet &\wtu_n\leq v_n \text{ when $\arg z\to\infty$.}
\end{array}
\end{equation}
(The first statement follows from \eqref{equation*12}
and the first statement of \eqref{equation*14}.)

\medskip
By the maximum principle, we have $\wtu_n\leq v_n$ in $A_n$.

\medskip

For any compact set $K$ of the set $\{z\in\CC\,:\,|z|\leq 1,\arg z\geq 0\}\setminus\{p_1,\cdots,p_{m}\}$,
the function $v_n$ is bounded by $C(K)$ on $K$.
(For $v_{3,n}$, use the last point of Lemma \ref{lemma3}.)
Then by symmetry, $u_n$ is bounded by $C(K)$ on $K\cup\overline{K}\cup
\sigma(K)\cup\sigma(\overline{K})$, where $\sigma$ denotes the inversion
$z\mapsto \overline{z}$.
Let
$$\Omega_{\infty}=\lim_{n\to\infty}\Omega_n=\CC\setminus\{\pm p_1,\cdots,\pm p_m\}.$$
Then $\wtu_n$ is bounded on compact subsets of $\Omega_{\infty}$.
By standard PDE theory, passing to a subsequence, $\wtu_n$ has a limit $\wtu$.
The convergence is the uniform smooth convergence on compact subsets of $\Omega_{\infty}$.
The limit has the following properties
\begin{itemize}
\item $\wtu$ is harmonic in $\Omega_{\infty}$. This follows from the first point of \eqref{equation*14}.
\item $\wtu(\overline{z})=-\wtu(z)$ and $\wtu(1/\overline{z})=\wtu(z)$.
\item $\wtu(z)\geq 0$ if $\arg z\geq 0$.
\end{itemize}
Note that either $\wtu\equiv 0$ or $\wtu$ is positive in $\arg z>0$.
Using the fact that $\log:\CC\to\C$ is biholomorphic, the following lemma tells us that $\wtu$ has the form given by equation \eqref{equation-limit-un}.
\begin{lemma}
\label{lemma-positive-harmonic}
Let $H$ be the upper half plane $\Im z> 0$ in $\C$. Let $u$ be a positive harmonic function in $H\setminus\{q_1,\cdots,q_m\}$ with boundary value $u=0$ on $\R$.
Then there exists non-negative constants $c_0,\cdots,c_m$ such that
$$u(z)=c_0\Im z-\sum_{i=1}^m c_i\log\left|\frac{z-q_i}{z-\overline{q_i}}\right|.$$
\end{lemma}
This lemma easily follows from the following two facts and the maximum principle:
\begin{itemize}
\item If $u$ is a positive harmonic function in the punctured disk $D(p,\varepsilon)\setminus\{p\}$
then $u(z)=-c\log|z-p|+v(z)$ where $v$ is harmonic in the disk.
(B\^ocher Theorem, \cite{HFT}, Theorem 3.9)
\item If $u$ is a positive harmonic function in $H$ with boundary value $0$ on $\R$ then
$u=c\Im z$. (\cite{HFT}, Theorem 7.22).
\end{itemize}
\medskip

To conclude the proof of Proposition \ref{proposition1},
it remains to compute the numbers $c_i$ for $1\leq i\leq m$.
Recall that $\phi_{i,n}$ is the vertical flux of $M_n$ on the graph of $f_n$
restricted to the circle $C(p_i,\varepsilon)$.
By Proposition \ref{proposition-flux2},
$$\phi_{i,n}=\Im\int_{C(p_i,\varepsilon)} (2f_{n,z}+O(t_n^2))dz
=\Im\int_{C(p_i,\varepsilon)} (-2u_{n,z}+O(t_n^2))dz.$$
Now
$$\lim \frac{|\log t_n|}{t_n} u_n = - c_i\log|z-p_i|+\text{harmonic} \quad\text{ near $p_i$}.$$
$$\lim\frac{|\log t_n|}{t_n} 2u_{n,z}=- \frac{c_i}{z-p_i} +\text{holomorphic}\quad\text{ near $p_i$}.$$
Hence by the Residue Theorem,
$$\lim \frac{|\log t_n|}{t_n}\phi_{i,n}
= 2\pi c_i.$$
This finishes the proof of Proposition \ref{proposition1}.
\cqfd
\medskip

As a corollary of the proof of Proposition \ref{proposition1}, we have
an estimate of $u_n$ that we will need in Section \ref{section45}.
For convenience, we state it here as a lemma.
\medskip

Fix some $\beta\in(0,\alpha)$ and let $A'_n\subset A_n$ be the domain defined as $A_n$ in
Definition \ref{definition*11}, replacing $\alpha$ by $\beta$, namely:
$A'_n$ is the set of points
$(z,\arg z)$ in $\CC$ which satisfy $t_n^{\beta}<|z|<1$ and $\arg z>0$,
minus the disks $D(p_{i,n},t_n^{\beta})$ for $1\leq i\leq N$.
\begin{lemma}
\label{lemma5} Assume that $p_1\neq 0$. Then
for $n$ large enough (depending only on $\beta$ and a lower bound on $|p_1|$), we have
$$u_n\leq(N+2)\frac{\beta}{\alpha} t_n\quad \mbox{ in } A'_n.$$
\end{lemma}
Recalling that $u_n<t_n/2$, this lemma is usefull when $\beta$ is small. We will use it
to get information about the level sets of $u_n$.

\medskip
Proof: as we have seen in the Proof of Proposition \ref{proposition1}, we have in $A_n$
\begin{equation}
\label{equation-un-vn}
u_n\leq \frac{t_n}{|\log t_n|} v_n=\frac{t_n}{|\log t_n|}(v_{1,n}+v_{2,n}+v_{3,n}).
\end{equation}
We need to estimate the functions $v_{1,n}$, $v_{2,n}$ and
$v_{3,n}$ in $A'_n$.
We have in $A_n$
$$v_{1,n}\leq C_1(C_2+N)t_n^{2-2\alpha}|\log t_n|=o(|\log t_n|).$$
By point 6 of Lemma \ref{lemma3}, we have in $A'_n$
$$v_{3,n}\leq \frac{1}{\alpha}|\log z|\leq \frac{1}{\alpha}|\log t_n^{\beta}|=\frac{\beta}{\alpha}|\log t_n|.$$
Regarding the function
$v_{2,n}$, we need to estimate each function $h_{p_{i,n}}$ in the domain  $A'_n$.
The function $h_{p_{i,n}}$ is harmonic in the domain
$$\{z\in\CC\,:\, \arg z>0,\,t_n^{\beta}< |z|<1\}\setminus D(p_{i,n},t_n^{\beta})$$
and goes to $0$ as $\arg z\to\infty$, so its maximum is on
the boundary. Since $h_{p_{i,n}}(1/\overline{z})=h_{p_{i,n}}(z)$, the maximum is not
on the circle $|z|=1$ (because it would be an interior maximum of $h_{p_{i,n}}$). Also $h_{p_{i,n}}=0$ on $\arg z=0$. Therefore,
the maximum is either on $|z|=t_n^{\beta}$ or on the circle $C(p_{i,n},t_n^{\beta})$.
On $|z|=t_n^{\beta}$, we have $h_{p_{i,n}}\to 0$ because $p_{i,n}$ is bounded away
from $0$. On the circle $C(p_{i,n},t_n^{\beta})$, we have for $n$ large
$$\log z-\log p_{i,n}\simeq \frac{1}{p_{i,n}}(z-p_{i,n})$$
$$|\log z-\log p_{i,n}|\geq \frac{t_n^{\beta}}{2|p_{i,n}|}.$$
Hence
$$-\log|\log z-\log p_{i,n}|\leq  \log (2 |p_{i,n}|)+\beta|\log t_n|.$$
Also,
$$\log|\log z-\log \overline{p_{i,n}}|\leq \log (|\log z|+|\log p_{i,n}|)\simeq \log(2|\log p_{i,n}|).$$
Since $|p_{i,n}|$ is bounded away from $0$, this gives for $n$ large enough
$$h_{p_{i,n}}\leq C+\beta|\log t_n|\quad\text{ in $A'_n$.}$$
Hence
$$v_{2,n}\leq C+N\frac{\beta}{\alpha} |\log t_n|.$$
Collecting all terms, we get, for $n$ large enough:
$$v_n\leq C+(N+1)\frac{\beta}{\alpha}|\log t_n|\leq (N+2)\frac{\beta}{\alpha}|\log t_n|\quad\text{ in $A'_n$}.$$
Using \eqref{equation-un-vn}, the lemma follows.
\cqfd
\subsection{Proof of Proposition \ref{proposition2}}
\label{section22}
We continue with the notation of the end of the previous section.
Fix some index $i$ and let $J=\{j\in [1,N]\,:\, p_j=p_i\}$.
Passing to a subsequence, we may assume that
$$r_{i,n}=\max\{r_{j,n}\,:\,j\in J\}.$$
(The numbers $r_{j,n}$ have been defined in Section \ref{section-setup}.)
Fix some positive $\varepsilon$ such that $|p_j-p_i|\geq 2\varepsilon$ for
$j\notin J$.

From Statement~$i.$ of Proposition~\ref{Setup},  we know that  near $p_{j,n}$ the surface $M_n$ is close to a vertical catenoid with  waist circle of radius $r_{j,n}$. More precisely,
$\frac{1}{r_{j,n}}(M_n-p_{i,n})$ converges to the standard catenoid
$$x_3=\cosh^{-1}\sqrt{x_1^2+x_2^2}.$$
Since the vertical flux of the standard catenoid is $2\pi$, we have
\begin{equation}
\label{estimate-phi}
\phi_{i,n}\simeq 2\pi \sum_{j\in J}r_{j,n}\leq 2\pi |J| r_{i,n}.
\end{equation}
Let
$$h_{j,n}=r_{j,n}\cosh^{-1}(2\lambda).$$
Observe that $h_{j,n}\ll t_n$.
Consider the intersection of $M_n$ with the plane at height $h_{j,n}$ and project it
on the horizontal plane. There is one component which is close
to the circle $C(p_{j,n},2\lambda r_{j,n})$. We call this component $\gamma_{j,n}$.
Observe that $\gamma_{j,n}\subset\Omega_n$ and $f_n=h_{j,n}$ on $\gamma_{j,n}$.
Let $D_{j,n}$ be the disk bounded by $\gamma_{j,n}$.
\medskip

We now estimate  $f_n$ on the circle $C(p_{i,n}, \varepsilon)$. 
By Proposition \ref{proposition1}, we know that
 $|u_n|=O(\frac{t_n}{|log t_n|})$. Hence $f_n=\frac{t_n}{2\pi}\arg z-u_n(z)\sim 
\frac{t_n}{2\pi}\arg z$ on  $C(p_{i,n}, \varepsilon)$.
Since $p_{i,n}$ is on the positive imaginary axis, $\arg z=\pi/2 +O(\varepsilon)$ on $C(p_{i,n}, \varepsilon)$. Hence$f_n(z)\sim \frac{t_n}{4}$ on $C(p_{i,n}, \varepsilon)$.
Consequently, the level set $f_n=\frac{t_n}{8}$ inside $\Omega_n\cap D(p_{i,n},\varepsilon)$ is a closed curve, possibly with several components.
We select the component
which encloses the point $p_{i,n}$ and call it $\Gamma_n$.
(Note that by choosing a very slightly different height, we may assume that $\Gamma_n$
is a regular curve). Let $D_n$ be the disk bounded by $\Gamma_n$.
Let
$$\Omega'_n=D_n\setminus\bigcup_{j\in J} D_{j,n}$$
Then $\Omega'_n\subset\Omega_n$. 
We are now able to apply the height estimate  of Appendix \ref{appendix-height}.
We apply Proposition \ref{proposition-height} with $r_1=\lambda r_{i,n}$,
$r_2=\varepsilon$, $h=t_n/8-h_{i,n}\simeq t_n/8$ and
$f$ equal to the
function $f_n(z-p_{i,n})-t_n/8$. (Observe that by Proposition~\ref{Setup}, Statement~${ii.}$, we may assume that $|\nabla f_n|\leq 1$. Also the fact that $\partial f_n/\partial\nu<0$ on $\gamma_{j,n}$
follows from the convergence to a catenoid.)
We obtain
$$\frac{t_n}{8}-h_{i,n}\leq \frac{\sqrt{2}}{\pi}\phi_{i,n} \log\frac{\varepsilon}{\lambda r_{i,n}}.$$
Using \eqref{estimate-phi}, this gives for $n$ large enough
\begin{equation}\label{FluxEstimateOne}
\frac{t_n}{9}\leq \frac{\sqrt{2}}{\pi}\phi_{i,n} \log\frac{2\pi|J|\varepsilon}{\lambda \phi_{i,n}}
\end{equation}
This implies
\begin{equation}\label{FluxEstimateTwo}\phi_{i,n} \geq \frac{2\pi |J|\varepsilon}{\lambda} t_n^2
\end{equation}
for $n$ large. To see this, suppose that $\phi_{i,n} < \frac{2\pi |J|\varepsilon}{\lambda} t_n^2$. Substituting in \eqref{FluxEstimateOne}, we get
$$ C_1 t_n\leq t_n^2 |\log t_n| $$
for some constant $C_1>0$. This is clearly a contradiction since $t_n|\log t_n|\rightarrow 0$.
Substitution of \eqref{FluxEstimateTwo} in \eqref{FluxEstimateOne} gives
$$\frac{t_n}{9}\leq \frac{2\sqrt{2}}{\pi} \phi_{i,n} |\log t_n|$$
which  implies that $\frac{|\log t_n|}{t_n}\phi_{i,n}$ is bounded below by a positive constant independent of $n$.
 Therefore, the coefficient $c_i$  defined in \eqref{equation-ci} is
 positive,  as desired.\cqfd
 \begin{remark}
 \label{remark-minoration-rin}
 Together with \eqref{estimate-phi}, this gives
 \begin{equation}
 \label{equation-minoration-rin}
 r_{i,n}\geq \frac{1}{36|J|\sqrt{2}}\frac{t_n}{|\log t_n|}
 \end{equation}
 for large $n$.
 This is a lower bound on the size of the largest catenoidal neck in the cluster corresponding to $p_i$.
 We have no lower bound for $r_{j,n}$ if $j\in J$, $j\neq i$. Conceptually, we could have
 $r_{j,n}=o(\frac{t_n}{|\log t_n|})$, although this seems unlikely. 
 \end{remark}
\subsection{Proof of Proposition \ref{proposition3}}
\label{section23}
Let $g_n=u_{n,z}.$
We have to prove
\[
\lim \left( \frac{\log t_n}{t_n}\right)^2
   \operatorname{Re} \int_{C(p_1,\epsilon)} \frac{2t_n}{4\pi i z} u_{n,z} (1-z)^2\,dz = 0,
\]
i.e., that
$$\Re\int_{C(p_1,\varepsilon)}\frac{1}{2iz}g_n(z)(1-z^2)dz=o\left(\frac{t_n}{(\log t_n)^2}\right).$$
Fix some $\alpha$ such that $0<\alpha<\frac{1}{2}$ and some small $\varepsilon>0$.
Let $J$ be the set of indices such that $p_j=p_1$.
Consider the domain
$$A_n=D(p_1,\varepsilon)-\bigcup_{j\in J} D(p_{j,n},t_n^{\alpha})\subset\Omega_n.$$
By Proposition \ref{proposition-schauder} in Appendix \ref{appendix-schauder}, we have in $A_n$
$$|g_{n,\overline{z}}|=\frac{1}{4}|\Delta u_n|=\frac{1}{4}|\Delta f_n|\leq Ct_n^{3-4\alpha}.$$
$$|\nabla f_n|\leq Ct_n^{1-\alpha}.$$
As the gradient of $t_n\arg z$ is $O(t_n)$ in $A_n$, this gives
$$|\nabla u_n|\leq Ct_n^{1-\alpha}.$$
Hence
\begin{equation}
\label{equation-gn}
|g_n|\leq Ct_n^{1-\alpha}.
\end{equation}
Proposition \ref{proposition-laurent} gives us the formula
$$g_n(z)=g^+(z)+
\sum_{j\in J}
g_j^-(z)+ \frac{1}{2\pi i}\int_{A_n}\frac{g_{n,\overline{z}}(w)}{w-z}dw\wedge\overline{dw}$$
where of course the functions $g^+$ and $g_j^-$ depend on $n$.
\begin{itemize}
\item The function $g^+$ is holomorphic in $D(p_1,\varepsilon)$ so does not contribute to the integral.
\item The last term is bounded by $Ct_n^{3-4\alpha}$. (The integral of $dw\wedge \overline{dw}/(w-z)$ is
uniformly convergent.) Therefore we need $3-4\alpha>1$, namely $\alpha<\frac{1}{2}$ 
so that the contribution of this term to the integral is $o(t_n/(\log t_n)^2)$.
\begin{remark} This is a crude estimate. The laplacian $\Delta u_n$ is bounded
by $Ct_n^3/d^4$, where $d$ is distance to the boundary. Integrating this estimate one
get that this term is less than $Ct_n^{3-2\alpha}$, which is better.
But one still needs $\alpha<1$ to ensure that this term is $o(t_n/(\log t_n)^2)$.
\end{remark}
\item Each function $f_j^-$ is expanded in series as in Proposition \ref{proposition-laurent}. By Proposition \ref{proposition-real-residue} in Appendix \ref{appendix-laurent},
each residue $a_{j,1}$ is real. Hence
\begin{equation}
\label{equation-aj1}
\Re\int_{C(p_1,\varepsilon)} \frac{1}{2iz} \frac{a_{j,1}}{z-p_{j,n}} (1-z^2)dz
=a_{j,1}\Re \left(\frac{2\pi i}{2i p_{j,n}}(1-p_{j,n}^2)\right)=0
\end{equation}
because $p_{j,n}$ is imaginary, so $a_{j,1}$ does not contribute to the integral.
\item It remains to estimate the coefficients $a_{j,k}$ for $k\geq 2$. Using
\eqref{equation-gn},
$$|a_{j,k}|=\left|\frac{1}{2\pi i}\int_{C(p_{j,n},t_n^{\alpha})}g_n(z)(z-p_{j,n})^{k-1}dz\right|
\leq C t_n^{1+(k-1)\alpha}$$
If $z\in C(p_1,\varepsilon)$, then $|z-p_{j,n}|\geq \varepsilon/2$, so
$$\left|
\sum_{k=2}^{\infty} a_{j,k} (z-p_{j,k})^{-k}\right|\leq C\sum_{k\geq 2} t_n^{1+(k-1)\alpha}
\left(\frac{2}{\epsilon}\right)^k\leq \frac{4C}{\varepsilon^2} t_n^{1+\alpha}\sum_{k=2}^{\infty}
\left(\frac{2t_n^{\alpha}}{\varepsilon}\right)^{k-2}.$$
The last sum converges because $\alpha>0$. Hence
the contribution of this term to the integral is $o(t_n/(\log t_n)^2)$ as desired.
\end{itemize}
\cqfd
\section{Case 2: $p_1=0$}
\label{section3}
In this case we make a blow up at the origin.
Let
$$R_n=\frac{1}{|p_{1,n}|}$$
(Here we assume again that the points $p_{i,n}$ are ordered by increasing imaginary
part as in Section \ref{section-setup}.)
Let $\whM_n=R_n M_n$. This is a helicoidal minimal surface in $\S^2(R_n)\times\R$ with
pitch $$\wht_n=R_n t_n.$$
By choice of $p_{1,n}$, we have $|p_{1,n}| >> t_n$, so $\lim \wht_n=0$. 
Let $\whOmega_n=R_n\Omega_n$. $\whM_n$ is the graph on $\whOmega_n$
of the function
$$\whf_n(z)=\frac{\wht_n}{2\pi}\arg z-\whu_n(z)$$
where
$$\whu_n(z)=R_n u_n(\frac{z}{R_n}).$$
Let $\whp_{i,n}=R_n p_{i,n}$. Passing to a subsequence
$$\whp_j=\lim\whp_{j,n}\in[i,\infty]$$
exists for $j\in[1,N]$ and we have $\whp_1=i$.
Let $m$ be the number of distinct, finite points amongst $\whp_1,\cdots,\whp_N$.
Relabel the points so that $\whp_1,\cdots,\whp_m$ are distinct and
$$1=\Im \whp_1<\Im\whp_2<\cdots<\Im \whp_m.$$
\begin{proposition}
\label{proposition4}
Passing to a subsequence,
$$\lim\frac{|\log \wht_n|}{\wht_n}\whu_n(z)=
\whc_0\arg z+\sum_{i=1}^{m} \whc_i h_{\whp_i}(z).$$
The convergence is the smooth uniform convergence on compact subsets of $\CC$ minus the points
$\pm\whp_i$, for $1\leq i\leq m$.
The numbers $\whc_i$ for $1\leq i\leq m$ are given by
$$\whc_i=\lim\frac{|\log\wht_n|}{\wht_n}\frac{\widehat{\phi}_{i,n}}{2\pi}$$
where $\widehat{\phi}_{i,n}$ is the vertical flux of $\whM_n$ on the graph of $\whf_n$ restricted
to the circle $C(\whp_i,\varepsilon)$, for some fixed small enough $\varepsilon$.
Moreover, $\whc_i>0$ for $1\leq i\leq m$.
\end{proposition}
This proposition is proved in Section \ref{section31}.
The proof is very similar to the proof of Proposition \ref{proposition1}, and Proposition \ref{proposition2} for the last statement.

\medskip
Fix some small $\varepsilon>0$.
Let $F_n$ be the flux of the Killing fields $\chi_Y$ on the circle $C(\whp_1,\varepsilon)$
on $\whM_n$. Since we are in $\S^2(R_n)\times\R$,
$$\chi_Y=\frac{i}{2}(1-\frac{z^2}{R_n^2}).$$
$$F_n=-\Im\int_{C(\whp_1,\varepsilon)}
2\left(\frac{\partial}{\partial z}\left(\frac{\wht_n}{2\pi}\arg z-\whu_n\right)\right)^2\frac{i}{2}\left(1-\frac{z^2}{R_n^2}\right)dz+O((\wht_n)^4).$$
Expand the square. As in Case 1, the cross product term can be neglected and
since $R_n\to\infty$:
\begin{proposition}
\label{proposition5}
$$\lim\left(\frac{\log \wht_n}{\wht_n}\right)^2 F_n=-\lim\left(\frac{\log \wht_n}{\wht_n}\right)^2\Re\int_{C(\whp_1,\varepsilon)} (\whu_{n,z})^2 dz$$
\end{proposition}
(Same proof as Proposition \ref{proposition3}).

\medskip

Assuming these results, we now prove
\begin{proposition}
Case 2 is impossible.
\end{proposition}
Proof:
Write
$\whp_j=i y_j$.
By the same computation as in Section \ref{section2}, we get
(the only difference is that there is no $(1-z^2)$ factor)
$$-\Re\int_{C(\whp_1,\varepsilon)}(\wtu_z)^2dz=
\frac{\pi}{2 y_1}\left(
\whc_1^2+\sum_{i=2}^m \frac{-2\pi^2\, \whc_1 \whc_i}{(\log y_1-\log y_i)|\log y_1-\log y_i+i\pi|^2}\right).$$
Again, since $y_i>y_1$ for $i\geq 2$, all terms are positive, contradiction.
\cqfd
\subsection{Proof of Proposition \ref{proposition4}}
\label{section31}
The setup of Proposition \ref{proposition4} is the same as Proposition \ref{proposition1}
except that we are in $\S^2(R_n)\times\R$ with $R_n\to\infty$
instead of $\S^2(1)\times\R$, and the pitch is $\wht_n$.
 Remember that $\lim \wht_n=0$.

\medskip

{\bf From now on forget all hats:} write $t_n$ instead of $\wht_n$,
$u_n$ instead of $\whu_n$, $p_{i,n}$ instead of $\whp_{i,n}$, etc...

\medskip
The proof of Proposition \ref{proposition4} is substantially the same as the proofs of Propositions \ref{proposition1} and \ref{proposition2}.
The main difference is that the equatorial circle $|z|=1$ becomes
$|z|=R_n$.
\begin{itemize}
\item The definition of the domain $A_n$ is the same with $|z|<1$ replaced by
$|z|<R_n$.
\item Lemma \ref{lemma1} is the same (recall that now $p_{i,n}$ means $\whp_{i,n}$).
\item Lemma \ref{lemma2} is the same. The last statement must be replaced by
$\partial g_n/\partial\nu\leq 0$ on $|z|=R$ for $R\geq 1$.
\item Lemma \ref{lemma3} is the same, we do not change the definition of the function
$H_t$. Instead of point 3, we need $\partial H_t/\partial\nu\geq 0$ on
$|z|=R$ for $R\geq 1$. This is true by the following computation:
$$\frac{\partial H_t}{\partial r}=\frac{2\log r(\log t)^2(|\log t|+\theta)}
{((\log t-\theta)^2+(\log r)^2)^2}.$$
\item The definition of the function $\wtu_n$ is the same, and it has the same properties,
except that the last point must be replaced by $\partial \wtu_n/\partial\nu=0$ on
$|z|=R_n$.
\item The definition of the function $v_{2,n}$ is the same (with $\whp_{i,n}$
in place of $p_{i,n}$), now it is symmetric
with respect to the circle $|z|=R_n$.
\item At the end, $K$ is a compact of the set $\{z\in\CC,\arg z\geq 0\}\setminus
\{\whp_1,\cdots,\whp_m\}.$
The fact that $v_{2,n}$ is uniformly bounded on $K$
requires some care, maybe, because some points $\whp_{i,n}$ are not
bounded: it is true by the fact that if $\arg z$ and $\arg p$ are positive, then
$$|\log z-\log p|\leq |\log z-\log\overline{p}|.$$
\item The proof of the last point is exactly the proof of Proposition \ref{proposition2},
working in $\S^2(R_n)\times\R$ instead of $\S^2(1)\times\R$.
\end{itemize}

\cqfd

\section{Case 3: $p_1=i$}
Note that in this case, all points $p_{j,n}$ converge to $i$, for $j\in[1,N]$.
\label{section4}
We distinguish two sub-cases:
\begin{itemize}
\item Case 3a: there exists $\beta>0$ such that $|p_{1,n}-i|\leq t_n^{\beta}$ for $n$ large enough,
\item Case 3b: for all $\beta>0$, $|p_{1,n}-i|\geq t_n^{\beta}$ for $n$ large enough.
\end{itemize}
(Here we assume again that the points $p_{i,n}$ are ordered by increasing imaginary
part as in Section \ref{section-setup}.)
Roughly speaking, in Case 3a, all points $p_{j,n}$ converge to $i$ quickly, whereas
in Case 3b, at least two ($p_{1,n}$ and $p_{N,n}$ by symmetry) converge to $i$ very slowly.
We will see in Proposition \ref{proposition-case3a} that $N=1$ and $p_{1,n}=i$
in Case 3a, and in Proposition \ref{proposition-case3b} that Case 3b is impossible.

\medskip
In both cases, we make a blowup at $i$ as follows :
Let $\varphi:\S^2\to\S^2$ be the rotation of angle $\pi/2$ which fixes the $Y$ circle and maps
$i$ to $0$. Explicitly, in our model of $\S^2(1)$
$$\varphi(z)=\frac{z-i}{1-iz},\qquad \varphi^{-1}(z)=\frac{z+i}{1+iz}.$$
It exchanges the equator $E$ and the great circle $X$.
$\varphi$ lifts in a natural way to an isometry $\widehat{\varphi}$ of $\S^2\times\R$.
We first apply the isometry $\widehat{\varphi}$ and
then we scale by $1/\mu_n$ where the ratio $\mu_n$ goes to zero and will be chosen later, depending on the case.
Let
$$\whM_n=\frac{1}{\mu_n}\widehat{\varphi}(M_n)\subset \S^2(1/\mu_n)\times\R,$$
$$\whOmega_n=\frac{1}{\mu_n}\varphi(\Omega),\qquad
\whp_{i,n}=\frac{1}{\mu_n}\varphi(p_{i,n}),\qquad
\wht_n=\frac{t_n}{\mu_n}.$$
The minimal surface $\whM_n$ is the graph over $\whOmega_n$ of the function
$$\whf_n(z)=\frac{1}{\mu_n}f_n(\varphi^{-1}(\mu_n z))=
\wht_n w_n(z)-\whu_n(z)$$
where
\begin{equation}
\label{equation-wn}
w_n(z)=\frac{1}{2\pi} \arg\left(\frac{\mu_n z+i}{1+i\mu_n z}\right)
\end{equation}
$$\whu_n(z)=\frac{1}{\mu_n} u_n(\varphi^{-1}(\mu_n z)).$$
\subsection{Case 3a}
\label{section41}
In this case, fix some positive number $\alpha$ such that
$\alpha<\min\{\beta,\frac{1}{8}\}$,
and take $\mu_n=t_n^{\alpha}$.
Then for all $j\in[1,N]$, $|p_{j,n}-i|=o(\mu_n)$, so $\lim \whp_{j,n}=0$.
\begin{proposition}
\label{proposition6}
In Case 3a, passing to a subsequence,
\begin{equation}
\label{EQ*1}
\lim\frac{|\log t_n|}{\wht_n}\left(\whu_n(z)-\whu_n(z_0)\right)=-c(\log|z|-\log |z_0|).
\end{equation}
The convergence is the uniform smooth convergence on compact subsets of $\C\setminus\{0\}$.
(Here $z_0$ is an arbitrary fixed nonzero complex number.)
The constant $c$ is positive.
\end{proposition}
The Proof is in Section \ref{section44}.
\begin{remark}
In fact
$$\lim \frac{|\log t_n|}{\wht_n}\whu_n(z)=\infty$$ for all $z$, so it is necessary to substract
something to get a finite limit.
Because of this, we believe it is not possible to prove this
proposition by a barrier argument as in the proof of Proposition \ref{proposition1}. Instead, we will
prove the convergence of the derivative $\whu_{n,z}$ using the Cauchy Pompeieu integral formula
for $C^1$ functions.
\end{remark}
We now prove
\begin{proposition}
\label{proposition-case3a}
In Case 3a, $N=1$.
\end{proposition}
Proof: From \eqref{equation-wn},
$$w_n(z)
=\frac{1}{2\pi}\left(\frac{\pi}{2}+O(\mu_n)\right)
=\frac{1}{4}(1+O(t_n^{\alpha})).$$
Since $\alpha>0$, $t_n^{\alpha}\to 0$
so using Equation \eqref{EQ*1} of
Proposition \ref{proposition6},
$$\whf_n(z)-\frac{\wht_n}{4}+\whu_n(z_0)\simeq  c\frac{\wht_n}{|\log t_n|}(\log |z|-\log|z_0|).$$
From this we conclude that for $n$ large enough,
the level curves of $\whf_n$ are convex.
Back to the original scale, we have found a horizontal convex curve $\gamma_n$
which encloses $N$ catenoidal necks and is invariant under reflection in the vertical cylinder $E\times\R$.
In particular, this curve $\gamma_n$ is a graph on each side of
$E\times\R$. Consider the domain on $M_n$
which is bounded by $\gamma_n$ and its symmetric image with respect to the $Y$-circle.
By Alexandrov reflection (see Appendix \ref{appendix-alexandrov}), this domain must be symmetric with respect to the vertical cylinder $E\times\R$ -- which we already know -- and must be a graph on each side of $E\times\R$. This implies that the centers of all necks must be on the circle $E$. But $E\cap Y^+$ is a single point.
Hence there is only one neck: $N=1$.\cqfd
\subsection{Case 3b}
\label{section42}
In this case we take $\mu_n=|p_{1,n}-i|$.
Passing to a subsequence, the limits
$$\whp_j=\lim \whp_{j,n}\in[\frac{-i}{2},\frac{i}{2}]$$
exist for all $j\in[1,N]$. Moreover, we have
$$
\whp_1=\frac{-i}{2} \text{ and } \whp_N=\frac{i}{2}.
$$
(The $\frac{1}{2}$ comes from the fact that the rotation $\varphi$ distorts euclidean lengths by the factor $\frac{1}{2}$ at $i$.)
Let $m$ be the number of distinct points amongst $\whp_1,\cdots,\whp_N$.
Observe that $m\geq 2$ because we know that $\whp_1$ and $\whp_N$
are distinct.
Relabel the points so that $\whp_1,\cdots,\whp_m$ are distinct and
$$\Im \whp_1<\Im \whp_2\cdots<\Im \whp_m.$$
\begin{proposition}
\label{proposition7}
In Case 3b, passing to a subsequence,
$$\wtu(z):=\lim\frac{|\log t_n|}{\wht_n}\left(\whu_n(z)-\whu_n(z_0)\right)=
\sum_{i=1}^m -\whc_i(\log|z-\whp_i|-\log|z_0-\whp_i|).$$
The convergence is the uniform smooth convergence on compact subsets of $\C$
minus the points $\whp_1,\cdots,\whp_m$.
(Here $z_0$ is an arbitrary fixed complex number different from these points.)
The constants $\whc_i$ are positive.
\end{proposition}
The proof of this proposition is in Section \ref{section45}.
\medskip

Fix some small number $\varepsilon>0$.
Let $F_n$ be the flux of the Killing field $\chi_Y$ on the circle $C(\whp_1,\varepsilon)$
on $\whM_n$.
Because of the scaling we are in $\S^2(1/\mu_n)\times\R$ so
$$\chi_Y(z)=\frac{i}{2}(1-\mu_n^2 z^2).$$
Hence  using Proposition \ref{proposition-flux2}
in Appendix \ref{appendix-flux},
\begin{equation}
\label{eq-david2}
F_n=-\Im\int_{C(\whp_1,\varepsilon)}
2\left(\wht_n w_{n,z}-\whu_{n,z}\right)^2\frac{i}{2}(1-\mu_n^2 z^2)+O((\wht_n )^4).
\end{equation}
Expand the square. Then as in Case 1, the cross-product term
can be neglected, so the leading term is the one involving $(\whu_{n,z})^2$
and since $\mu_n\to 0$:
\begin{proposition}
\label{proposition8}
\begin{equation}
\label{EQ*3}
\lim \left(\frac{\log t_n}{\wht_n}\right)^2 F_n=
-\lim \left(\frac{\log t_n}{\wht_n}\right)^2\Re\int_{C(\whp_1,\varepsilon)}(\whu_{n,z})^2 dz.
\end{equation}
\end{proposition}
This proposition is proved in Section \ref{section46}. The proof is similar to Proposition
\ref{proposition3}.

\medskip

We now prove
\begin{proposition}
\label{proposition-case3b}
Case 3b is impossible.
\end{proposition}
Proof: 
According to Lemma \ref{lemma*8}, the flux $F_n$ is equal to zero.
Hence the left-hand side of \eqref{EQ*3} is zero.
By Propositions \ref{proposition7} and \ref{proposition8},
\begin{equation}
\label{EQ*4}
0=-\Re\int_{C(\whp_1,\varepsilon)}(\wtu_z)^2
\end{equation}
On the other hand,
$$\wtu_z=-\sum_{i=1}^m \frac{\whc_i}{2(z-\whp_i)}$$
$$\Res_{\whp_1}(\wtu_z)^2=\frac{1}{2}\sum_{i=2}^m \frac{\whc_1 \whc_i}{\whp_1-\whp_i}.$$
Write $\whp_i=i y_i$, then
$$-\int_{C(\whp_1,\varepsilon)}(\wtu_z)^2
=-\pi\sum_{i=2}^m\frac{\whc_1 \whc_i}{y_1-y_i}.$$
Since $m\geq 2$, $y_1<y_i$ for all $i\geq 2$ and $\whc_i>0$ for all $i$ by Proposition
\ref{proposition7}, this is positive, contradicting \eqref{EQ*4}.
\cqfd

\medskip
This completes the proof of the main theorem, modulo the proof of
Propositions \ref{proposition6}, \ref{proposition7} and \ref{proposition8}, which
were used in the analysis of Cases 3a and 3b.
We prove these propositions in Sections \ref{section44}, \ref{section45} and \ref{section46} respectively, using an estimate that we prove in the next section.
\subsection{An estimate of $\int |\nabla u_n|$}
\label{section43}
By Proposition \ref{proposition1}, we have, since all points
$p_{j,n}$ converge to $i$,
$$\lim \frac{|\log t_n|}{t_n}u_n=c_0\arg z-c_1\log\left|\frac{\log z-\log i}
{\log z+\log i}\right|.$$
Moreover, $c_1$ is positive by Proposition \ref{proposition2}.
The convergence is the smooth convergence on compact subsets of
$\CC\setminus\{i,-i\}$.
From this we get, for fixed $\varepsilon>0$,
\begin{equation}
\label{equation-gradient-integral}
\int_{C(i,\varepsilon)}|\nabla u_n|\leq C \frac{t_n}{|\log t_n|}.
\end{equation}
Let $i\in[1,n]$ be the index such that $r_{i,n}=\max\{r_{j,n}\,:\,1\leq j\leq N\}$.
Let $\phi_n=\phi_{i,n}$ be the vertical flux of $M_n$ on the graph of $f_n$ restricted to
$C(p_{i,n},\varepsilon)$. By the last point of Proposition \ref{proposition1}, we have
$$\phi_n\leq C \frac{t_n}{|\log t_n|}$$
for some constant $C$.
We use Proposition \ref{proposition-height2} with $r_1=\lambda r_{i,n}$ and $r_2=\varepsilon$ as in
the proof of Proposition \ref{proposition2}, and
$$r'_1=(t_n)^{1/4},\qquad r'_2=(t_n)^{1/8}$$
The proposition tells us that for each $j\in[1,N]$,
there exists a number $r$, which we call
$r'_{j,n}$, such that
\begin{equation}
\label{eq-david1}
(t_n)^{1/4}\leq r'_{j,n}\leq (t_n)^{1/8}
\end{equation}
and
$$\int_{C(p_{j,n},r'_{j,n})\cap\Omega_n}|\nabla f_n|\leq
\sqrt{8}\phi_n\left(\log\frac{\varepsilon}{\lambda r_{i,n}}\right)^{1/2}
\left(\log\frac{(t_n)^{1/8}}{(t_n)^{1/4}}\right)^{-1/2}.$$
Using \eqref{equation-minoration-rin}, we have
$$\log \frac{\varepsilon}{\lambda r_{i,n}}
\leq \log \frac{\varepsilon |\log t_n|}{\lambda C_1 t_n}\leq C_2|\log t_n|$$
for some positive constants $C_1$ and $C_2$. This gives
$$\int_{C(p_{j,n},r'_{j,n})\cap\Omega_n}|\nabla f_n|\leq C\phi_n\leq C\frac{t_n}{|\log t_n|}.$$
Now since $|\nabla \arg z|\simeq 1$ near $i$,
$$\int_{C(p_{j,n}, r'_{j,n})} t_n|\nabla \arg z|\leq C t_n^{1+1/8}=o(\frac{t_n}{|\log t_n|}).$$
Hence
$$\int_{C(p_{j,n},r'_{j,n})\cap\Omega_n}|\nabla u_n|\leq C\frac{t_n}{|\log t_n|}.$$
Consider the domain
\begin{equation}
\label{eqUn}
U_n=D(i,\varepsilon)\setminus\bigcup_{j=1}^N \overline{D}(p_{j,n},r'_{j,n}).
\end{equation}
Since $r'_{j,n}\gg t_n\gg r_{j,n}$, we have $\overline{U_n}\subset\Omega_n$ and
\begin{equation}
\label{equation-distance-bord}
d(U_n,\partial \Omega_n)\geq \frac{1}{2} (t_n)^{1/4}.
\end{equation}
Also, since $\partial U_n\subset\Omega_n$,
$$\partial U_n\subset C(i,\varepsilon)\cup\bigcup_{j=1}^N (C(p_{j,n},r'_{j,n})\cap\Omega_n).$$
This implies
\begin{equation}
\label{equation-integral-gradient}
\int_{\partial U_n} |\nabla u_n|\leq C\frac{t_n}{|\log t_n|}.
\end{equation}
This is the estimate we will use in the next sections.

\subsection{Proof of Proposition \ref{proposition6} (Case 3a)}
\label{section44}
Let $\beta>0$ be the number given by the hypothesis of case 3a.
Recall that we have fixed some positive number $\alpha$ such that
$0<\alpha<\min\{\beta,\frac{1}{8}\}$, that $\mu_n=t_n^{\alpha}$,
$\wht_n=\frac{t_n}{\mu_n}$ and
$\varphi_n=\frac{1}{\mu_n}\varphi$.
Let $U_n$ be the domain defined in \eqref{eqUn}
and $\whU_n=\varphi_n(U_n)$.
Since $\mu_n\gg t_n^{1/8}\geq r'_{j,n}$
by \eqref{eq-david1},
we have
$$\lim \whU_n=\C^*.$$
Since $\varphi_n$ is conformal, we have, using \eqref{equation-integral-gradient}
(recall the definition of $\whu_n$ in \eqref{equation-wn})
$$
\int_{\partial \whU_n}|\nabla \whu_n|
=\int_{\partial\whU_n}\frac{1}{\mu_n}|\nabla (u_n\circ\varphi_n^{-1})|
=\frac{1}{\mu_n}\int_{\partial U_n}|\nabla u_n|\leq C\frac{t_n}{\mu_n|\log t_n|}=C\frac{\wht_n}{|\log t_n|}.
$$
Using \eqref{equation-distance-bord}, we have
$$d(\whU_n,\partial\whOmega_n)\geq \frac{(t_n)^{1/4}}{4\mu_n}.$$
By Proposition \ref{proposition-schauder} in Appendix \ref{appendix-schauder}
(Interior gradient and Laplacian estimate)
$$|\Delta \whu_n|=|\Delta \whf_n|\leq C \frac{(\wht_n)^3}
{((t_n)^{1/4}/(4\mu_n))^{4}}=C\mu_nt_n^{2}\qquad\mbox{ in } \whU_n.$$
Let
$$\wtu_n=\frac{|\log t_n|}{\wht_n} (\whu_n-\whu_n(z_0)).$$
Proposition \ref{proposition6} asserts that a subsequence of the $\tilde u_n$ converge to
$-c(\log |z| -\log |z_0|)$, where $c$ is a real positive constant.
By the above estimates,
\begin{equation}
\label{estimate-wtun}
\int_{\partial \whU_n}|\nabla \wtu_n|\leq C
\end{equation}
and
\begin{equation}
\label{equation-Delta-wtun}
|\Delta \wtu_n|\leq C\mu_n^2 t_n|\log t_n|\qquad\mbox{ in } \whU_n.
\end{equation}
Let $K$ be a compact set of $\C^*$.
For $n$ large enough, $K$ is included in $\whU_n$.
The Cauchy Pompeieu integral formula
(Equation \eqref{eq-pompeieu} in Appendix \ref{appendix-laurent})
gives for $\zeta\in K$
$$\wtu_{n,z}(\zeta)=\frac{1}{2\pi i}\int_{\partial \whU_n}\frac{\wtu_{n,z}(z)}{z-\zeta}dz
+\frac{1}{8\pi i}\int_{\whU_n}\frac{\Delta\wtu_n(z)}{z-\zeta} dz\wedge\overline{dz}.$$
We estimate each integral in the obvious way, using \eqref{estimate-wtun}
in the first line and \eqref{equation-Delta-wtun} in the third line:
$$\left|\int_{\partial\whU_n}\frac{\wtu_{n,z}}{z-\zeta}\right|\leq\frac{1}{d(\zeta,\partial \whU_n)}\int_{\partial\whU_n}
|\nabla\wtu_n|\leq \frac{C}{d(\zeta,\partial \whU_n)}\to \frac{C}{|\zeta|}.$$
$$\int_{\whU_n}\frac{dx\,dy}{|z-\zeta|}\leq
\int_{D(0,\varepsilon/\mu_n)}\frac{dx\,dy}{|z-\zeta|}
\leq 2\pi\int_{r=0}^{2\varepsilon/\mu_n}\frac{rdr}{r}
=4\pi\frac{\varepsilon}{\mu_n}.$$
$$\left|\int_{\whU_n}\frac{\Delta\wtu_n}{z-\zeta} dx\,dy\right|
\leq C\mu_n t_n|\log t_n|\to 0.$$
Hence for $n$ large enough, we have in $K$
$$|\wtu_{n,z}(\zeta)|\leq\frac{C}{|\zeta|}$$
for a constant $C$ independent of $K$.
Passing to a subsequence, $\wtu_{n,z}$ converges smoothly on compact sets
of $\C^*$ to a holomorphic function with a zero at $\infty$ and at most a simple pole at $0$.
(The fact that the limit is holomorphic follows from \eqref{equation-Delta-wtun}.)
Hence $$\lim\wtu_{n,z}=\frac{c}{2z}$$
for some constant $c$.
Recalling that $(\log|z|)_z =\frac{1}{2z}$, this gives \eqref{EQ*1} of Proposition \ref{proposition6}.
It remains to prove that $c>0$.
Let $\widehat{\phi}_n$ be the vertical flux on the closed curve of $\whM_n$ that is the graph of
$\whf_n$ over the circle $C(0,1)\subset\C^*$. Then by the same computation
as at the end of Section \ref{section-proof-proposition1},
$$\lim\frac{|\log t_n|}{\wht_n}\widehat{\phi}_n=2\pi c.$$
Now by scaling and homology invariance of the flux,
$\widehat{\phi}_n=\frac{\phi_{1,n}}{\mu_n}$, where $\phi_{1,n}$ is
the vertical flux on the closed curve of $M_n$ that is the graph of $f_n$ over
the circle $C(i,\varepsilon)$. Hence
$c=c_1$ and $c_1$ is positive by Proposition \ref{proposition1}.
\cqfd
\subsection{Proof of Proposition \ref{proposition7} (Case 3b)}
\label{section45}
Recall that in Case 3b, $\mu_n=|p_{1,n}-i|$ and for all $\beta>0$, 
$\mu_n\geq t_n^{\beta}$ for $n$ large enough.
Let $U_n$ be the domain defined in \eqref{eqUn}.
Since $\mu_n\gg t_n^{1/8}\geq r'_{j,n}$
by \eqref{eq-david1},
we have
$$\lim \whU_n=\C\setminus \{\whp_1,\cdots,\whp_m\}.$$
(Compare with Case 3a, where the limit is $\C^*$.)
Define again
$$\wtu_n=\frac{|\log t_n|}{\wht_n} (\whu_n-\whu_n(z_0)).$$
By the same argument as in Section \ref{section44} we obtain that
$\wtu_{n,z}$ converges on compact subsets of $\C\setminus\{\whp_1,\cdots,\whp_m\}$
to a meromorpic function with at most simple poles at $\whp_1,\cdots,\whp_m$ and
a zero at $\infty$, so
$$\lim\wtu_{n,z}=\sum_{i=1}^m\frac{\whc_i}{2(z-\whp_i)}.$$
It remains to prove that the numbers $\whc_1,\cdots,\whc_m$ are positive.
For $1\leq i\leq m$, let $\whphi_{i,n}$ be the vertical flux of $\whM_n$ on the
graph of $\whf_n$ restricted to the circle $C(\whp_i,\varepsilon)$.
Then by the computation at the end of Section \ref{section-proof-proposition1}, we have
$$\lim\frac{|\log t_n|}{\wht_n}\whphi_{i,n}=2\pi \whc_i.$$
We will prove that $\whc_i$ is positive by estimating the vertical flux using the
height estimate as in Section \ref{section23}.
Take $\beta=\frac{1}{18(N+2)}$ and let
$$B_n=\bigcup_{i=1}^N D(p_{i,n},t_n^{\beta}).$$
By Lemma \ref{lemma5} with $\alpha=\frac{1}{2}$, we have for $n$ large enough:
$$u_n\leq (N+2)\frac{\beta}{\alpha}=\frac{t_n}{9}
\quad\text{ in } D(i,\varepsilon)\setminus B_n.$$
(Lemma \ref{lemma5} gives us this estimate for $|z|\leq 1$. The result follows because
$u_n$ is symmetric with respect to the unit circle).
Consequently, the level set
$u_n=\frac{t_n}{8}$ is contained in $B_n$.
By the hypothesis of Case 3b, for $n$ large enough, $\mu_n\gg t_n^{\beta}$ so
 the disks $D(p_{i,n},t_n^{\beta})$
for $1\leq i\leq m$ are disjoint. Hence $B_n$ has at least $m$ components.
Let $\Gamma_{i,n}$ be the component of the level set $u_n=\frac{t_n}{8}$ which
encloses the point $p_{i,n}$ and $D_{i,n}$ the disk bounded by $\Gamma_{i,n}$. Then $D_{i,n}$ contains no other point
$p_{j,n}$ with $1\leq j\leq m$, $j\neq i$. (It might contain points $p_{j,n}$ with
$j>m$).
The proof of Proposition \ref{proposition2} in Section \ref{section22}
gives us a point $p_{j,n}\in D_{i,n}$ (with either $j=i$ or $j>m$ and $\whp_j=\whp_i$) such that
$$r_{j,n}\geq C\frac{t_n}{|\log t_n|}$$
for some positive constant $C$.
Scaling by $1/\mu_n$,
this implies that
$$\whphi_{i,n}\geq 2\pi \frac{C}{2}\frac{\wht_n}{|\log t_n|}.$$
Hence $\whc_i>0$.
\cqfd
\subsection{Proof of Proposition \ref{proposition8} (Case 3b)}
\label{section46}
Let $g_n=\whu_{n,z}$.
We have to prove that the cross-product term in \eqref{eq-david2} can
be neglected, namely:
$$\Re \int_{C(\whp_1,\varepsilon)} w_{n,z}(z) g_n(z)(1-\mu_n^2z^2)dz
=o\left(\frac{\wht_n}{(\log t_n)^2}\right).$$
The proof of this fact is the same as the proof of Proposition \ref{proposition3}
in Section \ref{section23}, with the following modifications:
\begin{itemize}
\item $\arg z$ is replaced by the function $w_n$ defined in \eqref{equation-wn}, so its
derivative $\frac{1}{2 iz}$ is replaced by $w_{n,z}$.
\item $1-z^2$ is replaced by $1-\mu_n^2 z^2$.
\item $t_n$, $u_n$, etc... now have hats: $\wht_n$, $\whu_n$, etc...
\item From
$$w_{n,z}=\frac{1}{4\pi i}\left(\frac{1}{\mu_n z+i}-\frac{i}{1+i\mu_n z}\right)$$
we deduce that $|w_{n,z}|$ is bounded in $D(\whp_1,\varepsilon)$ and
since $\whp_{j,n}\in i\R$, that
$w_{n,z}(\whp_{j,n})$ is real, which is what we need to ensure that the term $a_{j,1}$ does
not contribute to the integral (see \eqref{equation-aj1}).
\end{itemize}
\cqfd

\appendix
\section{Auxiliary results}
This appendix contains several results about minimal surfaces
in $\S^2\times\R$ that have been used in the proof of Theorem \ref{theorem4}.
Some of these results are true for minimal surfaces in the Riemannian product
$M\times\R$ where $(M,g)$ is a 2-dimensional Riemannian manifold.
These results are local, so we can assume without loss of generality that
$M$ is a domain $\Omega\subset\C$ equipped with a conformal
metric $g=\lambda^2|dz|^2$, where $\lambda$ is a smooth positive function
on $\overline{\Omega}$. Given a function $f$ on $\Omega$, the graph of
$f$ is a minimal surface in $M\times\R$ if it satisfies the minimal surface equation
\begin{equation}
\label{mse}
\div_g \frac{\nabla_g f}{W}=0\quad\mbox{ with }
W=\sqrt{1+||\nabla_g f||^2_g}
\end{equation}
where the subscript $g$ means that the quantity is computed with respect to the
metric $g$, so for instance
$$\nabla_g f =\lambda^{-2}\nabla f,\qquad
\div_g X=\lambda^{-2}\div(\lambda^2 X).$$
In coordinates, \eqref{mse} gives the equation
\begin{equation}
\label{msecoord}
(1+\lambda^{-2} f_y^2)f_{xx}+(1+\lambda^{-2} f_x^2)f_{yy}-2\lambda^{-2}
f_x f_y f_{xy}+(f_x^2+f_y^2)\left(
\frac{\lambda_x}{\lambda} f_x+\frac{\lambda_y}{\lambda} f_y\right)=0.
\end{equation}
Propositions \ref{proposition-schauder}, \ref{proposition-flux2}, \ref{proposition-height}
and \ref{proposition-height2}
will be formulated in this setup.
\subsection{Interior gradient and Laplacian estimate}
\label{appendix-schauder}
\begin{proposition}
\label{proposition-schauder}
Let $\Omega$ be a domain in $\C$ equipped with a smooth conformal metric
$g=\lambda^2 |dz|^2$.
 Let $f:\Omega\to\R$ be a solution of
the minimal surface equation \eqref{mse}. Assume that $|f|\leq t$ in $\Omega$ and
$||\nabla f||\leq 1$. Then
$$||\nabla f(z)||\leq \frac{Ct}{d(z)}$$
$$|\Delta f(z)|\leq \frac{Ct^3}{d(z)^4}$$
for all $z\in\Omega$ such that $d(z)\geq t$.
Here, $d(z)$ denotes the euclidean distance to the boundary of $\Omega$.
The gradient and Laplacian are for the euclidean metric.
The constant $C$ only depends on the diameter of $\Omega$ and on a bound
on $\lambda$, $\lambda^{-1}$ and its partial derivatives of first and second order.
\end{proposition}
Proof.
Let us write the minimal surface equation \eqref{msecoord} as $L(f)=0$,
where $L$ is a second order linear elliptic operator whose
coefficients depend on $f_x$ and $f_y$.
Theorem 12.4 in Gilbarg-Trudinger gives us a uniform constant $C$
and $\alpha>0$ such that (with Gilbarg-Trudinger notation)
$$[Df]_{\alpha}^{(1)}\leq C||f||_0\leq Ct.$$
If $d(z,\partial\Omega)\geq t$, this implies
$$[Df]_{\alpha}^{(0)}\leq \frac{Ct}{t}=C.$$
Then we have the required $C^{\alpha}$ estimates of the coefficients of $L$ to apply the interior Schauder estimate (Theorem 6.2 in Gilbarg-Trudinger):
$$|D^kf(z)|\leq \frac{C}{d(z)^k} ||f||_{0}\leq C\frac{t}{d(z)^k},\qquad k=0,1,2.$$
The minimal surface equation \eqref{msecoord} implies
$$|\Delta f|\leq C(|Df|^2 |D^2f|+ |Df|^3)\leq C\frac{t^3}{d^4}.$$
\cqfd
\subsection{Alexandrov moving planes}
\label{appendix-alexandrov}
We may use the Alexandrov reflection technique in $\S^2\times\R$ with the role of horizontal planes played by the level spheres $\S^2\times\{t\}$, and the role of
vertical planes played by a family of totally geodesic cylinders.
Specifically, let
$E\subset\S^2\times\{0\}$ be the closed geodesic that is the equator with respect to the antipodal points $O$, $O^*$, let $X\subset\S^2\times\{0\}$ be a geodesic passing through
$O$ and $O^*$, and define $E_{\theta}$ to be the rotation of $E=E_0$ through an angle 
$\theta$ around the poles $E\cap X$. The family of geodesic cylinders
$$E_{\theta}\times\R, \quad -\pi/2\leq\theta<\pi/2,$$
when restricted to the complement of $(E\cap X)\times\R$ is a foliation.

\begin{proposition}
\label{proposition-alexandrov}
Let $\Gamma=\gamma_1\cup\gamma_2$ with each $\gamma_i$ a $C^2$
Jordan curve in $\S^2\times\{t_i\}$, $t_1\neq t_2$, that is invariant under reflection
in $\Pi=E\times\R$. Suppose further that each component of $\gamma_i\setminus\Pi$
is a graph over $\Pi$ with locally bounded slope. Then any minimal surface
$\Sigma$ with $\partial\Sigma=\Gamma$ that is disjoint from at least one of the
vertical cylinders $E_{\theta}\times\R$, must be symmetric with respect to reflection in
$\Pi$, and each component of $\Sigma\setminus\Pi$ is a graph of locally bounded slope
over a domain in $\Pi$.
\end{proposition}
(Given a domain ${\mathcal O}\subset\Pi$ and a function 
$f:{\mathcal O}\to[-\pi/2,\pi/2)$, the graph of $f$ is the set of points
$\{\mbox{rot}_{f(p)} p\;:\; p\in{\mathcal O}\}$, where $\mbox{rot}_{\theta}$ is the rotational symmetry that takes
$\Pi$ to $E_{\theta}\times\R$.)
\medskip

The proof is the same as the classical proof for minimal surfaces in $\R^3$ using the maximum principle. (See for example Schoen \cite{schoen} Corollary 2.)

\subsection{Flux}
\label{appendix-flux}
Let $N$ be a Riemannian manifold, $M\subset N$ a minimal surface and $\chi$ a Killing field on $N$.
Let $\gamma$ be a closed curve on $M$ and $\mu$ be the conormal along $\gamma$.
Define
$$\mbox{Flux}_{\chi}(\gamma)=\int_{\gamma}\langle \mu,\chi\rangle ds.$$
It is well know that this only depends on the homology class of $\gamma$.
\begin{proposition}
\label{proposition-flux1}
In the case $N=\S^2(R)\times\R$,
the space of Killing fields is 4 dimensional. It is generated by
the vertical unit vector $\xi$, and the following three horizontal vectors fields:
$$\chi_X(z)=\frac{1}{2}(1+\frac{z^2}{R^2})$$
$$\chi_Y(z)=\frac{i}{2}(1-\frac{z^2}{R^2})$$
$$\chi_E(z)=\frac{iz}{R}$$
These vector fields are respectively unitary tangent to the great circles
$X$, $Y$ and $E$.
They are generated by the one-parameter families of rotations  about the poles whose equators are these great circles.
\end{proposition}
Proof: The isometry group of $\S^2(R)\times\R$ is well known to be 4-dimensional.
Recall that our model of $\S^2(R)$ is $\C\cup{\infty}$ with the conformal metric
$\frac{2R^2}{R^2+|z|^2}|dz|$.
By differentiating the 1-parameter group $z\mapsto e^{it}z$ of isometries of $\S^2$, we obtain the horizontal Killing field $\chi(z)=iz$, which suitably normalized gives $\chi_E$.
Let
$$\varphi(z)=\frac{Rz+iR^2}{iz+R}.$$
This corresponds, in our model of $\S^2(R)$, to the rotation about the $x$-axis of
angle $\pi/2$. It maps the great circle $E$ to the great circle $X$. We transport $\chi_E$ by this isometry
to get the Killing field $\chi_X$: a short computation gives
$$\chi_X(z)=\varphi_*\chi_E(z)=\varphi'(\varphi^{-1}(z))\chi_E(\varphi^{-1}(z))=\frac{z^2+R^2}{2R^2}.$$
Then we transport $\chi_X$ by the rotation $\psi(z)=iz$ to get the Killing field $\chi_Y$:
$$\chi_Y(z)=\psi_*\chi_X(z)=i \frac{(-iz)^2+R^2}{2R^2}.$$
\cqfd
\begin{proposition}
\label{proposition-flux2}
Let $\Omega\subset\C$ be a domain equipped with a conformal metric
$g=\lambda^2 |dz|^2$.
Let $f:\Omega\to\R$ be a solution of the minimal surface equation \eqref{mse}.
Let $\gamma$ be a closed, oriented curve in $\Omega$ and $\nu$ be the euclidean exterior normal vector along $\gamma$ (meaning that $(\gamma',\nu)$ is a negative orthonormal basis).
Let $M$ be the graph of $f$ and let $\widetilde{\gamma}$ be the closed curve in
$M$ that is the graph of $f$ over $\gamma$.
\begin{enumerate}
\item For the vertical unit vector $\xi$,
$$\mbox{Flux}_{\xi}(\widetilde{\gamma})=
\int_{\gamma}\frac{\langle \nabla f,\nu\rangle}{W}$$
where $W$ is defined in equation \eqref{mse}.
(Here the gradient, scalar product and line element are euclidean.)
If $||\nabla f||$ is small, this gives
$$\mbox{Flux}_{\xi}(\widetilde{\gamma})
=\Im\int_{\gamma} \left(2f_z + O(|f_z|^2)\right)dz$$
\item If $\chi$ is a horizontal Killing field, 
$$\mbox{Flux}_{\chi}(\widetilde{\gamma})=-\Im\int_{\gamma} \left(2(f_z)^2\chi(z) +O(|f_z|^4|)\right)dz.$$
\end{enumerate}
\end{proposition}
\noindent Proof:
Let $(N,g)$ be the Riemannian manifold $\Omega\times\R$ equipped with
the product metric $g=\lambda^2 |dz|^2+dt^2$.
Let $M$ be the graph of $f$, parametrized by
$$\psi(x,y)=(x,y,f(x,y)).$$
The unit normal vector to $M$ is
$$n=\frac{1}{W}\left(-\lambda^{-2}f_x,-\lambda^{-2}f_y,1\right).$$
Assume that $\gamma$ is given by some parametrization $t\mapsto\gamma(t)$, fix some
time $t$ and let $(X,Y)=\gamma'(t)$. Then
$$d\psi(\gamma')=(X,Y,X f_x+Y f_y)$$
is tangent to $\psi(\gamma)$ and its norm is $ds$, the line element on $M$.
We need to compute the conormal vector in $N$.
The linear map $\varphi:(T_p N,g)\to (\R^3,\mbox{euclidean})$ defined by
$$\varphi(u_1,u_2,u_3)=(\lambda u_1,\lambda u_2,u_3)$$
is an isometry. Let $u=(u_1,u_2,u_3)$ and $v=(v_1,v_2,v_3)$ be two orthogonal vectors
in $T_p N$. Let
$$w=\varphi^{-1}(\varphi(u)\wedge\varphi(v))=\left(\begin{array}{l}
u_2 v_3-u_3 v_2\\u_3 v_1-u_1 v_3\\ \lambda^2 (u_1 v_2 - u_2 v_1)\end{array}
\right).$$
Then $(u,v,w)$ is a direct orthogonal basis of $T_p N$ and
$||w||=||u||\;||v||$.
We use this with $u=d\psi(\gamma')$, $v=n$. Then $w=\mu \,ds$, where $\mu$
is the conormal to $\psi(\gamma')$. This gives
$$\mu\, ds=\frac{1}{W}\left(\begin{array}{l}
Y+\lambda^{-2}f_y(X f_x+Y f_y)\\
-X-\lambda^{-2} f_x(X f_x+Y f_y)\\
-f_y X+f_x Y\end{array}\right).$$
For the vertical unit vector $\xi=(0,0,1)$, this gives
$$\mbox{Flux}_\xi(\widetilde{\gamma})=
\int_{\gamma} \frac{-f_y  dx+ f_x dy}{W}=
\int_{\gamma} \frac{\langle \nabla f,\nu\rangle}{W}.$$
The second formula of point (1) follows from $W=1+O(||\nabla f||^2)$ and
$$\Im (2 f_z dz)=\Im \left((f_x -i f_y)(dx+idy)\right)=f_x dy -f_y dx.$$
To prove point (2), let $\chi$ be a horizontal Killing field, seen as a complex number. Then
$$\langle\chi,\mu ds\rangle_g
=\lambda^2\Re\left(
\frac{\chi}{W}(Y+iX+\lambda^{-2}(f_y+if_x)(X f_x+Y f_y)\right)$$
Hence
$$\mbox{Flux}_{\chi}(\widetilde{\gamma})=\Re\int_{\gamma}
\frac{\lambda^2\chi}{W}(dy+i\,dx)
+\frac{\chi}{W}(f_y+if_x)(f_x dx+f_y dy).$$
We then expand $1/W$ as a series
$$\frac{1}{W}=1-\frac{1}{2}\lambda^{-2}(f_x^2+f_y^2)+O(|\nabla f|^4).$$
This gives after some simplifications
$$\mbox{Flux}_{\chi}(\widetilde{\gamma})=\Re\int_{\gamma}
\lambda^2\chi(dy+i\,dx)
+\Re\int_{\gamma} \frac{i}{2}\chi(f_x-i f_y)^2 (dx+i \,dy)+O(|\nabla f|^4).$$
The second term is what we want. The first term, which does not depend on $f$,
vanishes. Indeed, if $f\equiv 0$ then $M$ is $\Omega\times\{0\}$ and the flux we are computing is zero (by homology invariance of the flux, say). 
\cqfd
\subsection{Height estimate}
\label{appendix-height}
The following proposition tells us that a minimal graph with small vertical flux cannot climb very high. It is the key to estimate from below the size of the catenoidal necks.
\begin{proposition}
\label{proposition-height}
Let $\Omega\subset\C$ be a domain that consists of a (topological) disk $D$ minus $n\geq 1$
topological disks $D_1,\cdots,D_n$ contained in $D$.
We denote by $\Gamma$ the boundary of $D$ and by
$\gamma_i$ the boundary of $D_i$. Assume that $D_1$ contains $D(0,r_1)$
and $D$ is contained in $D(0,r_2)$, for some numbers $0<r_1<r_2$.
(Here $r_1$, $r_2$ are euclidean lengths).
(See Figure \ref{figure-height-estimate}).

Assume that $\Omega$ is equipped with a conformal metric $g=\lambda^2 |dz|^2$.
Let $f:\Omega\to\R$ be a solution of the minimal surface equation \eqref{mse}.
Assume that
\begin{enumerate}
\item $f\equiv 0$ on $\Gamma$.
\item $f\equiv -h<0$ is constant on $\gamma_1$.
\item $f$ is constant on $\gamma_i$ for $2\leq i\leq n$, with
$-2h\leq f\leq 0$.
\item $\partial f/\partial \nu\leq 0$ on $\gamma_i$ for $1\leq i\leq n$.
\item $||\nabla_g f||_g\leq 1$ in $\Omega$
\end{enumerate}
Let $\phi$ be the vertical flux on $\Gamma$:
$$\phi=\int_{\Gamma} \frac{\langle\nabla f,\nu\rangle}{W}>0$$
Then
$$h\leq\frac{\sqrt{2}}{\pi}\phi\log\frac{r_2}{r_1}.$$
\end{proposition}
(Note that Hypothesis (4) is always satisfied if $f\equiv -h$ on all $\gamma_i$ by the
maximum principle.)
\begin{figure}
\begin{center}
\includegraphics[height=35mm]{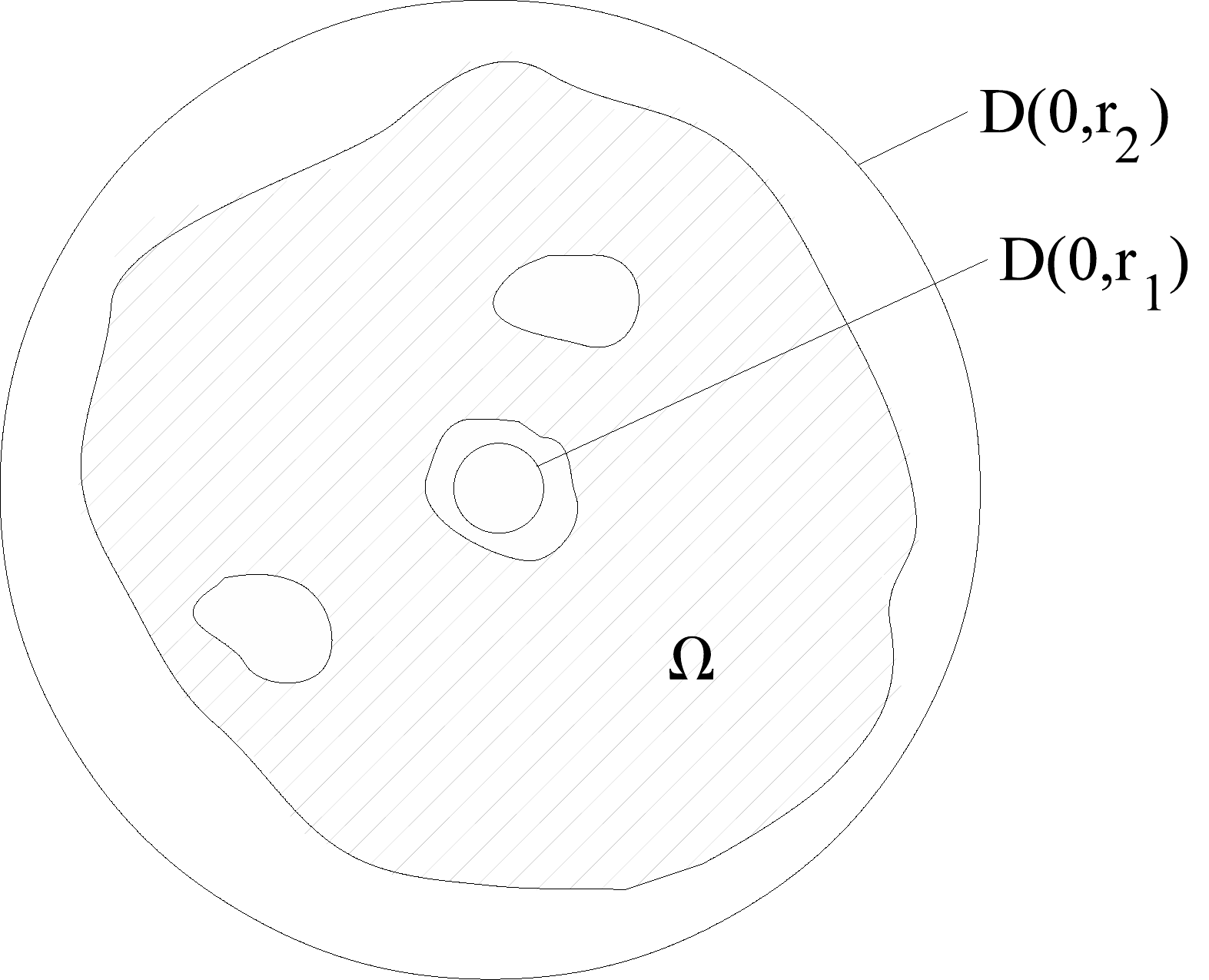}
\end{center}
\caption{}
\label{figure-height-estimate}
\end{figure}
\medskip

\noindent
Proof. Let $A$ be the annulus $D(0,r_2)\setminus D(0,r_1)$.
Write $|df|$ for the euclidean norm of the euclidean gradient of $f$.
Let $\rho$ be the function equal to $|d f|$
 on $\Omega$ and
$0$ on $\C\setminus\Omega$. Then
\begin{eqnarray*}
\iint_{A} \rho^2 dxdy
&=&
\iint_{\Omega} ||\nabla_g f||_g^2 d\mu_{g}\quad\mbox{ by conformal invariance of the energy}\\
&\leq &\sqrt{2}\iint_{\Omega} \langle\frac{\nabla_g f}{W},\nabla_g f\rangle_g d\mu_g
\quad\mbox{ because $W\leq \sqrt{2}$ by point (5)}\\
&=&\sqrt{2}\iint_{\Omega}\div_g(f\frac{\nabla_g f}{W})d\mu_g
\quad\mbox{ by the minimal surface equation \eqref{mse}}\\
&=&\sqrt{2}\int_{\partial \Omega} \frac{f}{W}\langle\nabla_g f,\nu_g\rangle_g ds_g
\quad\mbox{ by the divergence theorem}\\
&=&\sqrt{2}\int_{\partial \Omega} \frac{f}{W}\langle\nabla f,\nu\rangle
\quad\mbox{ where now all quantities are euclidean}\\
&=&\sqrt{2}\sum_{i=1}^n \int_{\gamma_i}\frac{f}{W}\langle\nabla f,\nu\rangle
\quad\mbox{ by point (1)}\\
&\leq& -2\sqrt{2}h \sum_{i=1}^n\int_{\gamma_i}\frac{\langle\nabla f,\nu\rangle}{W}
\quad\mbox{ by points (3) and (4)}
\end{eqnarray*}
Hence by homology invariance of the flux,
\begin{equation}
\label{eqq1}
\iint_{A} \rho^2 dxdy\leq 2\sqrt{2}h\phi
\end{equation}
Consider the ray from $r_1 e^{i\theta}$ to $r_2 e^{i\theta}$.
The integral of $df$ along this ray, intersected with $\Omega$, is equal to $h$.
(If the ray happens to enter one of the disks $D_i$,
then this is true because $f$ is constant on $\partial D_i$.)
Integrating for $\theta\in[0,2\pi]$ we get
\begin{eqnarray*}
2\pi h&\leq&\int_{r=r_1}^{r_2}\int_{\theta=0}^{2\pi} \rho dr\,d\theta\\
&=&\int_A \frac{\rho}{r} dx\,dy\\
&\leq& \left(\iint_A \rho^2 dx\,dy\right)^{1/2}\left(\iint_A \frac{1}{r^2}dx\,dy\right)^{1/2}
\quad\mbox{ by Cauchy Schwartz}\\
&\leq& \left(2\sqrt{2}h\phi\right)^{1/2} \left( 2\pi \log\frac{r_2}{r_1}\right)^{1/2}
\quad\mbox{ using \eqref{eqq1}}
\end{eqnarray*}
The proposition follows.
\cqfd

\medskip
The next proposition is useful to find circles on which we have a good estimate of $\int |df|$.
\begin{proposition}
\label{proposition-height2}
Under the same hypotheses as Proposition \ref{proposition-height},
consider some point $p\in\Omega$. Given $0<r'_1<r'_2$, there exists
$r\in[r'_1,r'_2]$ such that
$$\int_{C(p,r)\cap\Omega}|df|\leq \sqrt{8}\phi\left(\log\frac{r_2}{r_1}\right)^{1/2}
\left(\log\frac{r'_2}{r'_1}\right)^{-1/2}.$$
\end{proposition}
Proof: Consider the function
$$F(r)=\int_{C(p,r)\cap\Omega}|df|=\int_{\theta=0}^{2\pi} \rho(p+re^{i\theta})r d\theta.$$
Then
\begin{eqnarray*}
\lefteqn{\min_{r'_1\leq r\leq r'_2}F(r)\log\frac{r'_2}{r'_1}}\\
&\leq&
\int_{r=r'_1}^{r'_2} \frac{F(r)}{r} dr\\
&=&
\int_{r=r'_1}^{r'_2}\int_{\theta=0}^{2\pi} \frac{\rho(p+re^{i\theta})}{r} rdrd\theta\\
&\leq&
\left(\int_{r'_1}^{r'_2}\int_{0}^{2\pi} \rho(p+re^{i\theta})^2 rdrd\theta\right)^{1/2}
\left(\int_{r'_1}^{r'_2}\int_{0}^{2\pi} \frac{1}{r^2} rdrd\theta\right)^{1/2}\\
&\leq&\left(\int_A\rho^2 dx dy\right)^{1/2}\left(2\pi\log\frac{r'_2}{r'_1}\right)^{1/2}\\
&\leq&\left(8\phi^2\log\frac{r_2}{r_1}\log\frac{r'_2}{r'_1}\right)^{1/2}
\quad\mbox{ using \eqref{eqq1} and Proposition \ref{proposition-height}}
\end{eqnarray*}
The proposition follows.\cqfd
\subsection{A Laurent-type formula for $C^1$ functions}
\label{appendix-laurent}
 \begin{proposition}
 \label{proposition-laurent}
 Let $\Omega\subset\C$ be a domain of the form
 $$\Omega=D(0,R)\setminus \bigcup_{i=1}^n \overline{D}(p_i,r_i).$$
 Here we assume that the closed disks $\overline{D}(p_i,r_i)$ are disjoint and are included
 in $D(0,R)$.
 Let $f$ be a $C^1$ function on $\overline{\Omega}$. Then in $\Omega$,
 $$f(z)=f^+(z) + \sum_{i=1}^n f^-_i(z)+\frac{1}{2\pi i}\int_{\Omega}
 \frac{f_{\overline{z}}(w)}{w-z} dw\wedge\overline{dw}$$
 where
 $f^+$ is holomorphic in $D(0,R)$ and each $f^-_i$ is holomorphic
 in $\C\setminus \overline{D}(p_i,r_i)$.
 Moreover, these functions have the following series expansion
 $$f^+(z)=\sum_{k=0}^{\infty} a_k z^k\qquad \mbox{ with }
 a_k=\frac{1}{2\pi i}\int_{C(0,R)}\frac{f(z)}{z^{k+1}}dz$$
 $$f^-_i(z)=\sum_{k=1}^{\infty} \frac{a_{i,k}}{(z-p_i)^k}
 \qquad \mbox{ with } a_{i,k}=\frac{1}{2\pi i}\int_{C(p_i,r_i)}f(z)(z-p_i)^{k-1}dz$$
 The series converge uniformly in compact subsets of $\Omega$.
 \end{proposition}
 
 \begin{remark}
 This is the same as the Laurent series theorem except that there is a correction term which vanishes when $f$ is holomorphic. The integration circles in the formula
for $a_n$ and $a_{i,n}$ cannot be changed (as in the classical Laurent series theorem) since $f$ is not holomorphic.
\end{remark}
 Proof. By Cauchy Pompeieu integral formula for $C^1$ functions:
\begin{equation}
\label{eq-pompeieu}
 f(z)=\frac{1}{2\pi i}\int_{\partial \Omega} \frac{f(w)}{w-z}dw+\frac{1}{2\pi i}\int_{\Omega}\frac{f_{\overline{z}}(w)}{w-z}dw\wedge \overline{dw}.
\end{equation}
Define
$$f^+(z)=\frac{1}{2\pi i}\int_{C(0,R)}\frac{f(w)}{w-z}dw$$
$$f_i^-(z)=-\frac{1}{2\pi i}\int_{C(p_i,r_i)}\frac{f(w)}{w-z}dw$$
The function $f^+$ is holomorphic in $D(0,R)$. The function $f_i^-$ is holomorphic in
$\C\setminus D(p_i,r_i)$ and extends at $\infty$ with $f_i^-(\infty)=0$. These two functions are expanded in power series exactly as
in the proof of the classical theorem on Laurent series (see e.g. Conway \cite{conway} page 107).
\cqfd

\begin{proposition}
\label{proposition-real-residue}
Let $\Omega\subset\C$ be a domain as in Proposition \ref{proposition-laurent}.
Let $u:\Omega\to\R$ be a real-valued function of class $C^2$. Take
$f=\partial u/\partial z$. Then the coefficients $a_{i,1}$ which appear in
the conclusion of Proposition \ref{proposition-laurent}
are real.
\end{proposition}
\begin{proof}
\begin{align*}
\Im a_{i,1}
&=\frac{-1}{2\pi}\Re\int_{C(p_i,r_i)} u_z dz\\
&=\frac{-1}{4\pi}\int_{C(p_i,r_i)} u_z dz+u_{\overline{z}}d\overline{z}\quad
\mbox{ because $u$ is real valued}\\
&=\frac{-1}{4\pi}\int_{C(p_i,r_i)}du=0\quad
\mbox{ because $u$ is well defined in $\Omega$}
\end{align*}
\end{proof}
\subsection{Residue computation}
\label{appendix-residue}
\begin{proposition}
\label{proposition-residue}
 $$
  \Res_p (\log z - \log p)^{-1}=p,\qquad
 \Res_p(\frac{1-z^2}{4z^2})(\log z-\log p)^{-2} = -\frac{1+p^2}{4p}
.$$
\end{proposition}

\begin{proof}
$$\log z - \log p=\log \left (1+\frac{z-p}{p}\right)
=\frac{z-p}{p}-\frac{1}{2}\left(\frac{z-p}{p}\right)^2+O(z-p)^3$$
The first residue follows. Then
$$
(\log z-\log p)^{-2}=\left(\frac{z-p}{p}\right)^{-2}\left(1-\frac{1}{2}\left(\frac{z-p}{p}\right)
\right)^{-2}
=\frac{p^2}{(z-p)^2} + \frac{p}{z-p} + O(1).
$$
Let 
$$f(z)=\frac{1-z^2}{4z^2}= \frac1{4z^2} - \frac14.$$
Then
\begin{align*}
\Res_p\left(\frac{1-z^2}{4z^2}(\log z-\log p)^{-2}\right)
&=
\Res_p\left(\frac{f(z)p^2}{(z-p)^2}\right) + \Res_p\left( \frac{f(z)p}{(z-p)} \right)
\\
&=
f'(p)p^2 + f(p)p  \qquad(\text{by the Taylor expansion for $f$ at $p$})
\\
&=
-\frac1{2p^3}p^2 + \frac{1-p^2}{4p}
\\
&=
-\frac{1+p^2}{4p}.
\end{align*}
\end{proof}


\begin{thebibliography}{1}
\bibitem{HFT} S. Axler, P. Bourdon, W. Ramey: Harmonic Function Theory. Springer Verlag, New York (1992).
\bibitem{bobenko}
Alexander I. Bobenko:
Helicoids with handles and {B}aker-{A}khiezer spinors.
Math. Z. 229:1, 9--29 (1998).
\bibitem{conway} John B. Conway: Functions of One Complex Variable, Second Edition. Graduate Texts in Mathematics 11. Springer Verlag.
\bibitem{partI} David Hoffman, Martin Traizet, Brian White: Helicoidal minimal surfaces of prescribed genus, I. Preprint (2013).
\bibitem{schoen} Rick Schoen: Uniqueness, Symmetry, and Embeddedness of Minimal Surfaces. J. of  Differential Geometry 18, 791--809 (1983).
\bibitem{schmies} Markus Schmies:
Computational methods for Riemann surfaces and helicoids with handles.
Thesis, University of Berlin (2005).
\bibitem{traizet1} Martin Traizet: A balancing condition for weak limits of minimal surfaces. Comment. Math. Helv. 79, 798--825 (2004).
\bibitem{traizet2} Martin Traizet: On minimal surfaces bounded by two convex curves in parallel planes. Comment. Math. Helv. 85, 39--71 (2010).

\end{thebibliography}
\end{document}